\documentclass[10pt]{amsart}
\usepackage{graphicx}
\usepackage{amscd}
\usepackage{amsmath}
\usepackage{amsthm}
\usepackage{amsfonts}
\usepackage{amssymb}
\usepackage{mathrsfs}
\usepackage{bm}
\usepackage{enumerate}
\usepackage{amsrefs}
\usepackage{xcolor}
\usepackage[colorlinks, citecolor=blue, linkcolor=red, pdfstartview=FitB]{hyperref}
\usepackage[latin1]{inputenc}

\usepackage{tikz}
\usepackage{tikz-cd}
\usepackage{caption}
\usetikzlibrary{decorations.markings}
\usepackage{lipsum}
\usepackage{mathrsfs}

\usepackage{marginnote}	
\usepackage{soul} 			

\newcommand{\bB}{{\mathbb{B}}}
\newcommand{\bC}{{\mathbb{C}}}
\newcommand{\bD}{{\mathbb{D}}}

\newcommand{\bN}{{\mathbb{N}}}

\newcommand{\bT}{{\mathbb{T}}}

  \newcommand{\A}{{\mathcal{A}}}

  \newcommand{\F}{{\mathcal{F}}}
  
\renewcommand{\H}{{\mathcal{H}}}

  \newcommand{\K}{{\mathcal{K}}}  

  \newcommand{\M}{{\mathcal{M}}}
  \newcommand{\N}{{\mathcal{N}}}
\renewcommand{\O}{{\mathcal{O}}}

\renewcommand{\S}{{\mathcal{S}}}
  
  \newcommand{\U}{{\mathcal{U}}}

\newcommand{\fA}{{\mathfrak{A}}}

\newcommand{\fB}{{\mathfrak{B}}}

\newcommand{\fD}{{\mathfrak{D}}}

\newcommand{\fJ}{{\mathfrak{J}}}
\newcommand{\fK}{{\mathfrak{K}}}

\newcommand{\fT}{{\mathfrak{T}}}

\newcommand{\rA}{\mathrm{A}}

\newcommand{\rC}{\mathrm{C}}

\newcommand{\rP}{\mathrm{P}}

\newcommand{\rU}{\mathrm{U}}

\renewcommand{\phi}{\varphi}

\newcommand{\upchi}{{\raise.35ex\hbox{$\chi$}}}

\newcommand{\ol}{\overline}

\newtheorem{lemma}{Lemma}[section]
\newtheorem{theorem}[lemma]{Theorem}
\newtheorem{proposition}[lemma]{Proposition}
\newtheorem{corollary}[lemma]{Corollary}

\theoremstyle{definition}
\newtheorem{definition}[lemma]{Definition}
\newtheorem{remark}[lemma]{Remark}

\newtheorem{example}{Example}

\date{\today}
\author{Ian Thompson}

\address{Department of Mathematics, University of Manitoba, Winnipeg, Manitoba, Canada R3T 2N2}

\email{thompsoi@myumanitoba.ca\vspace{-2ex}}
\thanks{The author was partially supported by an NSERC CGS-D Scholarship.}

\title[Approximate unique extension property]{An approximate unique extension property for completely positive maps}
\begin{document}
\begin{abstract}
We study the closure of the unitary orbit of a given point in the non-commutative Choquet boundary of a unital operator space with respect to the topology of pointwise norm convergence. This may be described more extensively as the $*$-representations of the $\rC^*$-envelope that are approximately unitarily equivalent to one that possesses the unique extension property. Although these $*$-representations do not necessarily have the unique extension property themselves, we show that their unital completely positive extensions display significant restrictions. When the underlying operator space is separable, this allows us to connect our work to Arveson's hyperrigidity conjecture. Finally, as an application, we reformulate the classical \v{S}a\v{s}kin Theorem and Arveson's essential normality conjecture.
\end{abstract}
\maketitle

\section{Introduction}\label{S:Introduction}

The Weyl von-Neumann Theorem characterizes normal operators on a separable Hilbert space as the diagonal operators modulo the compacts. In \cite{voiculescu1976non}, Voiculescu unveiled a vast non-commutative generalization regarding $*$-representations of separable $\rC^*$-algebras on separable Hilbert spaces. Namely, every $*$-representation $\pi:\fA\rightarrow B(\H_\pi)$ on a separable Hilbert space lies in the pointwise norm closed unitary orbit of some direct sum of irreducible $*$-representations. That is, there is a collection of irreducible $*$-representations $\{ \pi_i:\fA\rightarrow B(\H_i): i\in I\}$ and a sequence of unitary operators $u_n:\bigoplus_{i\in I}\H_i\rightarrow\H_\pi$ such that \[\lim_{n\rightarrow\infty}\left\| u_n^*\pi(t)u_n -\left(\bigoplus_{i\in I}\pi_i\right)(t)\right\| = 0, \ \ \ \ \ \ \ t\in\fA.\]Moreover, we may assume that $u_n^*\pi(t)u_n$ is equal to the direct sum of irreducible $*$-representations modulo the compacts for each $n$. Within \cite{voiculescu1976non}, there are several results collectively labelled as `Voiculescu's Theorem.' For our purposes, by Voiculescu's Theorem we mean that $*$-representations of a separable $\rC^*$-algebra acting on a separable Hilbert space possess the same point-norm closed unitary equivalence class precisely when the $*$-representations have pointwise equal rank. 

In \cite[Theorem 3.14]{hadwin1981nonseparable}, Hadwin provided a non-separable adaption of Voiculescu's Theorem, which we now describe. For this, recall that the rank of a Hilbert space operator will refer to the dimension of the closure of the range. We say that two $*$-representations $\pi,\sigma$ of a $\rC^*$-algebra $\fA$ are \emph{approximately unitarily equivalent} if there is a net of unitary operators $u_\beta:\H_\sigma\rightarrow\H_\pi$ such that $\| u_\beta^* \pi(t)u_\beta - \sigma(t)\|\rightarrow0$ for each $t\in \fA$. Whenever $\fA$ and $\H_\pi$ are separable, the net may be replaced by a sequence. Then, Hadwin showed that $\pi$ and $\sigma$ are approximately unitarily equivalent if and only if $\text{rank}(\pi(t)) = \text{rank}(\sigma(t))$ for every $t\in\fA$. When $\H_\pi$ fails to be separable, one may imagine difficulties in capturing cardinality but our results do not require such intricacies. As a consequence to Hadwin's Theorem, one can generally expect that the approximate unitary equivalence class of a $*$-representation is quite large. Indeed, if $\pi$ and $\sigma$ are $*$-representations with the same kernel, then $\pi^{(\kappa)}$ and $\sigma^{(\kappa)}$ are approximately unitarily equivalent for a suitably large choice of cardinal $\kappa$.

Our work focuses on studying the approximate unitary equivalence class of a particularly significant class of $*$-representations that are associated with a unital operator space. To this end, an \emph{operator space} is a subspace $\M\subset B(\H)$ and, when $\M$ is self-adjoint and unital, $\M$ is referred to as an \emph{operator system}. We let the smallest $\rC^*$-algebra containing a unital operator space $\M$ in $B(\H)$ be denoted by $\rC^*(\M)$. A \emph{$\rC^*$-cover} is a pair $(\fA,\iota)$ where $\iota:\M\rightarrow B(\K)$ is a unital completely isometric linear map and $\fA = \rC^*(\iota(\M))$. The \emph{$\rC^*$-envelope} of $\M$ is a $\rC^*$-cover $(\rC^*_e(\M), \varepsilon)$ that satisfies the following universal property: whenever $(\fA,\iota)$ is a $\rC^*$-cover of $\M$, there exists a $*$-representation $\theta:\fA\rightarrow\rC^*_e(\M)$ satisfying $\theta\circ\iota=\varepsilon$. Existence of the $\rC^*$-envelope is non-trivial and was originally shown in \cite{hamana1979injective}. Alternatively, a large program of Arveson \cite{arveson1969subalgebras} showed that the $\rC^*$-envelope always exists due to the existence of sufficiently many boundary representations for $\M$ \cite{arveson2008noncommutative},\cite{davidson2015choquet},\cite{dritschel2005boundary},\cite{muhly1998algebraic}. We say that a $*$-representation $\pi:\rC^*(\M)\rightarrow B(\H_\pi)$ has the \emph{unique extension property} with respect to $\M$ if, whenever $\psi:\rC^*(\M)\rightarrow B(\H_\pi)$ is a unital completely positive map satisfying $\psi\mid_\M = \pi\mid_\M$, then we have that $\psi = \pi$. If, in addition, $\pi$ is irreducible, then $\pi$ is said to be a \emph{boundary representation} for $\M$. Arveson's program, as well as many subsequent developments \cite{arveson1998subalgebras},\cite{arveson2011noncommutative},\cite{clouatre2021finite},\cite{clouatre2018multiplier},\cite{clouatre2020finite},\cite{clouatre2022minimal},\cite{davidson2019noncommutative},\cite{davidson2021choquet},\cite{davidson2022strongly},\cite{kakariadis2013vsilov},\cite{kennedy2015essential},\cite{kennedy2021noncommutative},\cite{kleski2014korovkin}, have demonstrated that boundary representations encode many underlying features of a unital operator space and may be thought of as a non-commutative Choquet boundary \cite{arveson2008noncommutative}. We take our main inspiration from one of these developments \cite{arveson2011noncommutative}.

Motivated by classical results of Korovkin and \v{S}a\v{s}kin, Arveson proposed a non-commutative approximation theory that is encoded in terms of boundary representations \cite{arveson2011noncommutative}. Initially, Korovkin had proven a rigidity result concerning sequences of positive maps with domain and codomain equal to $\rC[0,1]$. \v{S}a\v{s}kin had then demonstrated that Korovkin's result is an example of a particular unital operator space $\M\subset\rC[0,1]$ satisfying the property that all irreducible $*$-representations of $\rC[0,1]$ are boundary representations for $\M$ \cite{vsavskin1967mil}. Explicitly, \v{S}a\v{s}kin had showed that, given a compact metric space $X$, every irreducible $*$-representation for $\rC(X)$ is a boundary representation for a unital operator space $\M\subset\rC(X)$ if and only if, whenever $\psi_n:\rC(X)\rightarrow\rC(X)$ is a sequence of unital positive maps such that $\|\psi_n(g)-g\|\rightarrow0$ for every $g\in\M$, then we have that $\|\psi_n(f)-f\|\rightarrow0$ for every $f\in\rC(X)$. Decades later, Arveson provided a non-commutative adaptation of this phenomenon \cite[Theorem 2.1]{arveson2011noncommutative}. Given a separable unital operator space $\M\subset B(\H)$, Arveson showed that every $*$-representation of $\rC^*(\M)$ has the unique extension property with respect to $\M$ if and only if, whenever $\omega:\rC^*(\M)\rightarrow B(\K)$ is a faithful $*$-representation and $\psi_n: B(\K)\rightarrow B(\K)$ is a sequence of unital completely positive maps satisfying \[\lim_{n\rightarrow\infty}\| \psi_n(\omega(m))-\omega(m)\| = 0, \ \ \ \ \ \ \ m\in\M,\]then we necessarily have that \[ \lim_{n\rightarrow\infty}\| \psi_n(\omega(t))-\omega(t)\|= 0, \ \ \ \ \ \ \  t\in\rC^*(\M).\]A unital operator space $\M$ with the property that all $*$-representations of $\rC^*(\M)$ have the unique extension property is said to be \emph{hyperrigid}. Unfortunately, Arveson was unable to frame hyperrigidity in terms of irreducible $*$-representations. Accordingly, Arveson conjectured that every irreducible $*$-representation of $\rC^*(\M)$ is a boundary representation for a separable unital operator space $\M$ if and only if every $*$-representation of $\rC^*(\M)$ has the unique extension property with respect to $\M$. Despite significant effort \cite{arveson2011noncommutative},\cite{clouatre2018non},\cite{clouatre2018unperforated},\cite{clouatre2018multiplier},\cite{davidson2021choquet},\cite{katsoulis2020hyperrigidity},\cite{kennedy2015essential},\cite{kim2021hyperrigidity},\cite{kleski2014korovkin},\cite{salomon2019hyperrigid}, an answer to Arveson's hyperrigidity conjecture is still unknown, even in the commutative case. Indeed, the reader should note that it is not clear whether \v{S}a\v{s}kin's result coincides with the commutative case of Arveson's result due to the difference in ranges for the completely positive maps.

For the non-trivial direction of Arveson's conjecture, suppose that every irreducible $*$-representation of $\rC^*(\M)$ is a boundary representation for $\M$. It is well-known that the unique extension property is preserved under direct sums \cite[Proposition 4.4]{arveson2011noncommutative}. So, Voiculescu's Theorem implies that every $*$-representation of $\rC^*(\M)$ that acts on a separable Hilbert space is approximately unitarily equivalent to a $*$-representation with the unique extension property. If Arveson's conjecture were true, then these $*$-representations would necessarily have the unique extension property. Accordingly, the main purpose of our work is to study those $*$-representations that are approximately unitarily equivalent to a $*$-representation with the unique extension property.

In Section \ref{S:AUE_UEP}, we record preliminary information and basic observations regarding the $*$-representations of $\rC^*(\M)$ that are approximately unitarily equivalent to a $*$-representation with the unique extension property. In particular, we address whether \emph{irreducible} $*$-representations that are approximately unitarily equivalent to a boundary representation are necessarily boundary representations. In the literature, there are a couple of examples demonstrating that this question has a negative answer (Example \ref{E:CuntzInf} and surrounding discussion) and we record the details of one such example.

In Section \ref{S:ApproxUEP}, we study a variation on the unique extension property. Our first main development (Theorem \ref{T:UEPimplyAUEP}) shows that this variation is inherited among those $*$-representations that are approximately unitarily equivalent to a $*$-representation with the unique extension property.

\begin{theorem}\label{T:A}
Let $\M\subset B(\H)$ be a unital operator space and $\pi:\rC^*(\M)\rightarrow B(\H_\pi)$ be a $*$-representation that is approximately unitarily equivalent to a $*$-representation with the unique extension property. If $\psi:\rC^*(\M)\rightarrow B(\H_\pi)$ is a unital completely positive map satisfying $\psi\mid_\M = \pi\mid_\M$, then there is a net of unitary operators $u_\beta\in B(\H_\pi)$ such that $u_\beta^*\psi(t)u_\beta\rightarrow\pi(t)$ in the weak operator topology for every $t\in\rC^*(\M)$.
\end{theorem}

We say that a $*$-representation of $\rC^*(\M)$ that satisfies the conclusion of Theorem \ref{T:A} has the \emph{approximate unique extension property} with respect to $\M$. Throughout Section \ref{S:ApproxUEP}, we study the class of $*$-representations with the approximate unique extension property and develop a collection of structural properties that are satisfied by these $*$-representations (Propositions \ref{P:AUEP_aue}, \ref{P:AUEPSum}, and \ref{P:AUEPker}). Consequently, this allows us to relate our work to Arveson's hyperrigidity conjecture (Theorems \ref{T:AUESummary} and \ref{T:Compact}). For this, recall that a $\rC^*$-algebra $\fA$ is \emph{postliminal} if the image of every irreducible $*$-representation for $\fA$ contains the ideal of compact operators.

\begin{theorem}\label{T:B}
Let $\M\subset B(\H)$ be a separable unital operator space. Consider the following statements:\begin{enumerate}[\rm (i)]
\item The operator space $\M$ is hyperrigid.
\item Every irreducible $*$-representation of $\rC^*(\M)$ is a boundary representation.
\item Every irreducible $*$-representation of $\rC^*(\M)$ is approximately unitarily equivalent to a boundary representation.
\item Every $*$-representation of $\rC^*(\M)$ has the approximate unique extension property.
\end{enumerate}Then, we have that {\rm (i)}$\Rightarrow${\rm (ii)}$\Rightarrow${\rm (iii)}$\Rightarrow${\rm (iv)} and {\rm (iii)}$\not\Rightarrow${\rm (ii)}. Moreover, if $\rC^*(\M)$ is postliminal, then we have that {\rm (iv)}$\Rightarrow${\rm (ii)}.
\end{theorem}

Since {\rm (iii)}$\not\Rightarrow${\rm (ii)} in general, our results will imply all of the same restrictions to a class of unital operator spaces that do not satisfy the property that every irreducible $*$-representation is a boundary representation. Thus, these operator spaces have the property that the space of completely positive extensions for every $*$-representation is incredibly restrictive, yet there are irreducible $*$-representations that fail to be boundary representations (Example \ref{E:CuntzInf}).

We now unveil the main result of our work, which provides a restriction on the space of completely positive extensions of $*$-representations with the approximate unique extension property (Theorems \ref{T:AUEPsequence} and \ref{T:GammaUEP}).

\begin{theorem}\label{T:C}
Let $\M\subset B(\H)$ be a unital operator space and $\pi:\rC^*(\M)\rightarrow B(\H_\pi)$ be a $*$-representation that possesses the approximate unique extension property. Suppose that $\psi:\rC^*(\M)\rightarrow B(\H_\pi)$ is a unital completely positive map satisfying $\psi\mid_\M = \pi\mid_\M$. Then, the following statements hold.\begin{enumerate}[\rm (i)]
\item There is a homomorphic conditional expectation $\Gamma_\psi:\rC^*(\text{Im}\psi)\rightarrow \text{Im}\pi$ satisfying $\Gamma_\psi\circ\psi = \pi$.
\item The conditional expectation $\Gamma_\psi$ has the unique extension property with respect to $\text{Im}\psi$.
\item We have that $\ker\Gamma_\psi$ is the smallest closed two-sided ideal of $\rC^*(\text{Im}\psi)$ that contains $\text{Im}(\psi-\pi)$. In particular, $\psi = \pi$ if and only if $\Gamma_\psi$ is isometric.
\end{enumerate}
\end{theorem}

As a consequence, Theorem \ref{T:C} demonstrates that there is a split exact sequence of $\rC^*$-algebras given by\[\begin{tikzcd}
0 \arrow{r} & \langle\langle\text{Im}(\psi-\pi)\rangle\rangle \arrow{r} & \rC^*(\text{Im}\psi) \arrow{r} & \text{Im}\pi \arrow{r} & 0
\end{tikzcd}\]where the second term is the smallest closed two-sided ideal of $\rC^*(\text{Im}\psi)$ that contains $\text{Im}(\psi-\pi)$. The subsequent results in Section \ref{S:SplitSeq} are consequences to the above developments. In particular, this allows us to apply our machinery to obtain new results on the structure of completely positive extensions of arbitrary isometric $*$-representations of the $\rC^*$-envelope (Corollary \ref{C:AUEPisom}).

In Section \ref{S:Examples}, we consider two classes of concrete examples where our results may apply. The first is a reformulation of \v{S}a\v{s}kin's Theorem (Proposition \ref{P:Saskin}) that addresses commutativity and the approximate unique extension property. For the second application, note that Voiculescu's Theorem and Arveson's program guarantee that the collection of $*$-representations that possess the approximate unique extension property is always quite large. Therefore, one would expect that requiring a single $*$-representation to have the approximate unique extension property to be a weak assumption. Nevertheless, we apply our machinery to operator systems arising from quotient modules of multipliers on the Drury-Arveson space and find that hyperrigidity is encoded in a single $*$-representation possessing the approximate unique extension property (Theorem \ref{T:HREssNorm}). Due to work of Kennedy and Shalit \cite{kennedy2015essential}, our development allows us to reframe the essential normality conjecture of Arveson.

{\bf Acknowledgements.} The author would like to thank their advisor, Rapha\"{e}l Clou\^{a}tre, for many helpful discussions during the preparation of this manuscript.

\section{Approximate unitary equivalence and the unique extension property}\label{S:AUE_UEP}

In this section, we present a few preliminary facts as well as the fundamental relationship between the unique extension property and its preservation under approximate unitary equivalence. The results within Subsections \ref{SS:aue_uep} and \ref{SS:asymp_uep} are likely known to experts, but a condensed treatment is unavailable.

\subsection{Preliminaries}\label{SS:prelim}

First, we record some well-known facts in the literature. We start with a central tool for determining when two $*$-representations are approximately unitarily equivalent. This was first shown in \cite[Theorem 3.14]{hadwin1981nonseparable}.

\begin{theorem}\label{T:HadwinAUE}
Let $\fA$ be a $\rC^*$-algebra. Suppose that $\pi:\fA\rightarrow B(\H_\pi)$ and $\sigma:\fA\rightarrow B(\H_\sigma)$ are $*$-representations. Then, the following statements are equivalent.\begin{enumerate}[\rm (i)]
\item The $*$-representations $\pi$ and $\sigma$ are approximately unitarily equivalent.
\item We have that $\text{rank}(\pi(t)) = \text{rank}(\sigma(t))$ for every $t\in\fA$.
\end{enumerate}
\end{theorem}

Next, we record a few basic facts regarding $*$-representations with the unique extension property.

\begin{lemma}\label{L:UEP}
Let $\M\subset B(\H)$ be a unital operator space and $\{ \pi_i: \rC^*(\M)\rightarrow B(\H_i) : i\in I\}$ be a collection of $*$-representations. Then, the following statements are equivalent.\begin{enumerate}[\rm (i)]
\item For each $i\in I$, $\pi_i$ has the unique extension property with respect to $\M$.
\item We have that $\bigoplus_{i\in I}\pi _i$ has the unique extension property with respect to $\M$.
\end{enumerate} 
\end{lemma}

\begin{proof}
See \cite[Proposition 4.4]{arveson2011noncommutative} and \cite[Lemma 2.8]{clouatre2020finite}.
\end{proof}

Finally, we require a couple known facts on hyperrigidity and boundary representations. In particular, statement {\rm (iii)} will imply that, in the interest of the hyperrigidity conjecture, it is of no harm to assume that $\M$ generates its $\rC^*$-envelope.

\begin{theorem}\label{T:HR}
Let $\M\subset B(\H)$ be a unital operator space and $(\rC^*_e(\M), \varepsilon)$ denote the $\rC^*$-envelope of $\M$. Let $\theta:\rC^*(\M)\rightarrow \rC^*_e(\M)$ denote the $*$-representation that satisfies $\theta\mid_\M = \varepsilon$. Then, the following statements hold.\begin{enumerate}[\rm (i)]
\item There is an isometric $*$-representation of $\rC^*_e(\M)$ that possesses the unique extension property with respect to $\varepsilon(\M)$.
\item If $\beta:\rC^*(\M)\rightarrow B(\H_\beta)$ is a boundary representation for $\M$, then there is a $*$-representation $\widetilde{\beta}:\rC^*_e(\M)\rightarrow B(\H_\beta)$ such that $\beta=\widetilde{\beta}\circ\theta$.
\item If every irreducible $*$-representation of $\rC^*(\M)$ is a boundary representation for $\M$, then $\rC^*(\M)$ is the $\rC^*$-envelope of $\M$.
\end{enumerate}
\end{theorem}

\begin{proof}
{\rm (i)}: By \cite[Lemma 2.11]{clouatre2020finite}, there is a collection $\Delta$ consisting of boundary representations for $\varepsilon(\M)$, that are defined on $\rC^*_e(\M)$, such that $\bigoplus_{\delta\in\Delta}\delta$ is isometric. Thus, by Lemma \ref{L:UEP}, $\bigoplus_{\delta\in\Delta}\delta$ is an isometric $*$-representation of $\rC^*_e(\M)$ that has the unique extension property with respect to $\varepsilon(\M)$.

{\rm (ii)}: By \cite[Theorem 7.1]{arveson2008noncommutative} and \cite[Theorem 3.3]{davidson2015choquet}, there is a collection $\Sigma$ consisting of boundary representations for $\M$, that are defined on $\rC^*(\M)$, such that $\bigoplus_{\sigma\in\Sigma}\sigma$ is completely isometric on $\M$ (alternatively, see \cite[Theorem 2.10]{clouatre2020finite}). Define a closed two-sided ideal of $\rC^*(\M)$ by \[\fJ_\Sigma=\bigcap_{\sigma\in\Sigma}\ker\sigma.\]By \cite[Theorem 2.2.3]{arveson1969subalgebras}, the ideal $\fJ_\Sigma$ coincides with the intersection of the kernels of all boundary representations for $\M$ that are defined on $\rC^*(\M)$. Then, by \cite[Proposition 2.2.4 and Theorem 2.2.5]{arveson1969subalgebras}, we have that $\rC^*(\M)/\fJ_\Sigma$ is the $\rC^*$-envelope of $\M$. Since $\fJ_\Sigma\subset\ker\beta$, statement {\rm (ii)} is immediate.

{\rm (iii)}: Again, let $\fJ_\Sigma$ denote the intersection of the kernels of all boundary representations for $\M$ that are defined on $\rC^*(\M)$. By \cite[Proposition 2.2.4 and Theorem 2.2.5]{arveson1969subalgebras}, $\rC^*(\M)/\fJ_\Sigma$ is the $\rC^*$-envelope of $\M$. So {\rm (iii)} is immediate since $\fJ_\Sigma=\{0\}$.
\end{proof}

We remark that the closed two-sided ideal $\fJ_\Sigma$ constructed in the proof of Theorem \ref{T:HR} is referred to as the the \emph{Shilov ideal} of $\M$ in $\rC^*(\M)$. Since $\rC^*(\M)/\fJ_\Sigma$ is the $\rC^*$-envelope of $\M$, the collection of all boundary representations for $\M$ is typically thought of as the \emph{non-commutative Choquet boundary} for $\M$.

\subsection{Approximate unitary equivalence and the unique extension property}\label{SS:aue_uep}

For our investigation, we start by showing that the collection of $*$-representations that possess the unique extension property is rarely closed under approximate unitary equivalence.

\begin{proposition}\label{P:hr_aue}
Let $\M\subset B(\H)$ be a unital operator space and assume that $\rC^*(\M)$ is the $\rC^*$-envelope of $\M$. Then, the following statements are equivalent.\begin{enumerate}[\rm (i)]
\item The operator space $\M$ is hyperrigid.
\item Whenever $\pi:\rC^*(\M)\rightarrow B(\H_\pi)$ is a $*$-representation that is approximately unitarily equivalent to a $*$-representation with the unique extension property, then $\pi$ has the unique extension property.
\end{enumerate}
\end{proposition}

\begin{proof}
It is trivial that {\rm (i)}$\Rightarrow${\rm (ii)}. For the converse, let $\pi:\rC^*(\M)\rightarrow B(\H_\pi)$ be a $*$-representation. We show that $\pi$ has the unique extension property. Since $\M$ generates its $\rC^*$-envelope, by Lemma \ref{L:UEP} and Theorem \ref{T:HR} {\rm (i)}, we may assume that $\pi$ is injective. Define an infinite cardinal $\kappa = \max\{\aleph_0, \dim\H_\pi\}$. By considering $\pi^{(\kappa)}$ if necessary, we may additionally assume that $\text{rank}(\pi(t))=\kappa$ for each non-zero $t\in\rC^*(\M)$.

Now, by Theorem \ref{T:HR} {\rm (i)}, there is an injective $*$-representation $\sigma$ of $\rC^*(\M)$ that possesses the unique extension property with respect to $\M$. Again, we may assume that $\text{rank}(\sigma(t))=\kappa$ for every non-zero $t\in\rC^*(\M)$. By Theorem \ref{T:HadwinAUE}, we have that $\pi$ and $\sigma$ are approximately unitarily equivalent. Whence, $\pi$ has the unique extension property with respect to $\M$ by assumption.
\end{proof}

\begin{remark}\label{R:DirInt}
Assume that $\rC^*_e(\M)$ is separable and postliminal. To study hyperrigidity, the reader may find it desirable to consider how the unique extension property behaves with respect to direct integrals as this machinery can be used to classify multiplicity-free $*$-representations up to unitary equivalence \cite[Theorem 4.3.4]{arveson2012invitation}. However, due to \cite[Corollary 6.3]{hadwin1981nonseparable} and Lemma \ref{L:UEP}, the unique extension property is stable under direct integrals whenever it is stable under approximate unitary equivalence. We opt to study the latter, but remark that there has been some progress using direct integrals \cite{arveson2011noncommutative},\cite{kleski2014korovkin}.
\end{remark}

Now, we discuss approximate unitary equivalence of \emph{irreducible} $*$-representations with the unique extension property. As the next two examples demonstrate, the answer to this question is more subtle than the previous one.

We first note that a variation of Proposition \ref{P:hr_aue} fails even in the classical setting of function spaces. For this, recall that postliminal $\rC^*$-algebras satisfy the property that irreducible $*$-representations that possess the same kernel are necessarily unitarily equivalent \cite[Corollary 4.1.10]{diximier1977c}. Thus, whenever the $\rC^*$-envelope is postliminal, the non-commutative Choquet boundary is automatically stable under approximate unitary equivalence.

\begin{example}\label{E:all_BR_aue}
There is a compact metric space $X$ and a unital subalgbera $\A\subset\rC(X)$ that separates points in $X$ with the property that the Choquet boundary is a proper subset of the Shilov boundary \cite[pg. 42]{phelps2001lectures}. In other words, not every irreducible $*$-representation of $\rC^*_e(\A)$ is a boundary representation for $\A$ \cite[Theorem II.11.3]{gamelin2005uniform}.

On the other hand, suppose that $\pi$ is a boundary representation for $\A$. If $\sigma$ is an irreducible $*$-representation of $\rC^*_e(\A)$ that is approximately unitarily equivalent to $\pi$, then $\pi$ and $\sigma$ are unitarily equivalent by the previous discussion. Whence, $\sigma$ is also a boundary representation for $\A$. So, the boundary representations for $\A$ are stable under approximate unitary equivalence, yet not every irreducible $*$-representation of $\rC^*_e(\A)$ is a boundary representation for $\A$.
\end{example}

The above conclusion can be achieved for any other unital operator space that fails to be hyperrigid and has a postliminal $\rC^*$-envelope. However, when we remove the postliminal assumption, the non-commutative Choquet boundary is not necessarily closed under approximate unitary equivalence. Two examples that demonstrate the failure of this preservation are the operator system generated by a pair of freely independent semicircular operators \cite[Example 6.6.3]{davidson2019noncommutative} and the infinite Cuntz algebra \cite[Continuation of Example 2.7]{muhly1998tensor}. We recall the details of the latter example.

\begin{example}\label{E:CuntzInf}
Let $\O_\infty$ denote the infinite Cuntz algebra with generators $(V_n)_{n\in\bN},$ and consider the algebra $\A$ that is generated by $\{V_n : n\in\bN\}$. As $\rC^*(\A)\cong\O_\infty$ is simple, we have that all irreducible $*$-representations of $\O_\infty$ are approximately unitarily equivalent.

However, up to unitary equivalence, there is a unique irreducible $*$-representation of $\O_\infty$ that is not a boundary representation for $\A$. Indeed, the $*$-representations of $\O_\infty$ are in one-to-one correspondence with sequences of isometries $(W_n)_{n\geq0}$ satisfying $\sum_{n=1}^\infty W_nW_n^*\leq I$. An irreducible $*$-representation is a boundary representation for $\A$ if and only if the associated sequence of isometries satisfies $\sum_{n=1}^\infty W_nW_n^*= I$ (see \cite{muhly1998tensor} for details). By work of Popescu \cite[Theorem 1.3]{popescu1989isometric}, up to unitary equivalence, there is a unique such irreducible $*$-representation with $\sum_{n=1}^\infty W_nW_n^*\neq I$.
\end{example}

We remark that the operator system generated by a pair of freely independent semicircular operators possibly demonstrates an opposite extrema to that of the infinite Cuntz algebra. Indeed, within \cite[Example 6.6.3]{davidson2019noncommutative}, an uncountable class of irreducible $*$-representations were constructed where all but a finite number were boundary representations. Simultaneously, one can obtain that all irreducible $*$-representations of the $\rC^*$-envelope are approximately unitarily equivalent for the same reasoning as in Example \ref{E:CuntzInf}. On the other hand, in Example \ref{E:CuntzInf}, it was found that there is a unique irreducible $*$-representation that fails to be a boundary representation.

\subsection{Asymptotics of the unique extension property}\label{SS:asymp_uep}

The unique extension property enjoys a well-known characterization in the framework of pointwise convergent sequences. The separable version can be found within \cite[Proposition 2.2]{kleski2014korovkin}. We record the non-separable version, which is obtained by replacing sequences with nets.

\begin{proposition}\label{P:UEPasym}
Let $\M\subset B(\H)$ be a unital operator space and let $\pi: \rC^*(\M)\rightarrow B(\H_\pi)$ be a $*$-representation. Then, the following statements are equivalent.\begin{enumerate}[\rm (i)]
\item The $*$-representation $\pi$ has the unique extension property with respect to $\M$.
\item If $\psi_\alpha: \rC^*(\M)\rightarrow B(\H_\pi)$ is a net of unital completely positive maps such that $\psi_\alpha(m)\rightarrow\pi(m)$ in the weak operator topology for every $m\in\M$, then $\psi_\alpha(t)\rightarrow\pi(t)$ in the weak operator topology for every $t\in\rC^*(\M).$
\item If $\psi_\alpha: \rC^*(\M)\rightarrow B(\H_\pi)$ is a net of unital completely positive maps such that $\psi_\alpha(m)\rightarrow\pi(m)$ in the strong operator topology for every $m\in\M$, then $\psi_\alpha(t)\rightarrow\pi(t)$ in the strong operator topology for every $t\in\rC^*(\M).$
\end{enumerate}
\end{proposition}

We present a reinterpretation of Proposition \ref{P:UEPasym} that allows us to connect the unique extension property with a notion of approximate unitary equivalence for unital completely positive extensions. We will return to this philosophy in future sections.

\begin{proposition}\label{P:UEPbackwardsApprox}
Let $\M\subset B(\H)$ be a unital operator space and $\pi: \rC^*(\M)\rightarrow B(\H_\pi)$ be a $*$-representation. Then, the following statements are equivalent.\begin{enumerate}[\rm (i)]
\item The $*$-representation $\pi$ has the unique extension property with respect to $\M$.
\item If $\psi: \rC^*(\M)\rightarrow B(\H_\pi)$ is a unital completely positive map satisfying $\psi\mid_\M =\pi\mid_\M$, then there exists a net of unitary operators $w_\beta\in B(\H_\pi)$ such that $w_\beta^*\pi(t)w_\beta\rightarrow\psi(t)$ in the strong operator topology for every $t\in \rC^*(\M)$.
\end{enumerate}
\end{proposition}

\begin{proof}
${\rm (i)}\Rightarrow {\rm (ii)}:$ In this case, $\psi = \pi$ and so we may take $w_\beta$ to be the identity operator on $\H_\pi$.

${\rm (ii)}\Rightarrow{\rm (i)}:$ Let $\psi: \rC^*(\M)\rightarrow B(\H_\pi)$ be a unital completely positive map satisfying $\psi\mid_\M = \pi\mid_\M$ and let $w_\beta\in B(\H_\pi)$ be a net of unitary operators such that $w_\beta^*\pi(t)w_\beta\rightarrow\psi(t)$ in the strong operator topology for each $t\in \rC^*(\M)$. As multiplication over bounded subsets is continuous in the strong operator topology, and as the net $(w_\beta\pi(\cdot)w_\beta)_\beta$ consists of $*$-representations, we may conclude that $\psi(st) =  \psi(s)\psi(t)$ for all $s,t\in\rC^*(\M).$ Since $\psi\mid_\M=\pi\mid_\M$, we then obtain that $\psi = \pi$.
\end{proof}

As we shall see in Section \ref{S:ApproxUEP}, the equivalence of Proposition \ref{P:UEPbackwardsApprox} is no longer true if one were to interchange the roles of the maps $\psi$ and $\pi$ in statement {\rm (ii)}.

\section{Approximate unique extension property}\label{S:ApproxUEP}

In accordance with studying the approximate unitary equivalence class of a $*$-representation possessing the unique extension property, we focus our attention on a variation of the unique extension property.

\begin{definition}\label{D:ApproxUEP}
Let $\M\subset B(\H)$ be a unital operator space and let $\pi: \rC^*(\M)\rightarrow B(\H_\pi)$ be a $*$-representation. We say that $\pi$ has the \emph{approximate unique extension property} with respect to $\M$ if, whenever $\psi: \rC^*(\M)\rightarrow B(\H_\pi)$ is a unital completely positive map satisfying $\psi\mid_\M = \pi\mid_\M$, then there is a net of unitaries $u_\beta\in B(\H_\pi)$ such that $u_\beta^*\psi(t)u_\beta\rightarrow\pi(t)$ in the weak operator topology for every $t\in\rC^*(\M)$.
\end{definition}

When there is no confusion, we will refrain from referencing the operator space and simply state that a $*$-representation possesses the approximate unique extension property. Moreover, when $\M$ and $\H_\pi$ are separable, then the net of unitaries may be replaced by a sequence.

We caution the reader that the approximate unique extension property differs from the characterization of the unique extension property that was seen within Proposition \ref{P:UEPbackwardsApprox} {\rm (ii)}. For example, we will see that this follows from Theorem \ref{T:UEPimplyAUEP}. Additionally, we remark that Definition \ref{D:ApproxUEP} is somewhat reminiscent of different conditions seen in the literature but remains largely independent \cite[Theorem 5.2]{davidson2017dilations},\cite[Theorem 2.4]{hadwin1987completely}.

Recall that boundary representations factor through the $\rC^*$-envelope by Theorem \ref{T:HR} {\rm (ii)}. Since the approximate unique extension property appears to be a significant constraint on the space of completely positive extensions, one may wonder whether similar behaviour is present. While the author is unaware of such a restriction, we start with a fact that is sufficient for our purposes.

\begin{lemma}\label{L:ApproxUEPquotient}
Let $\M\subset B(\H)$ be a unital operator space and $\fA\subset B(\K)$ be a $\rC^*$-algebra. Let $q: \rC^*(\M)\rightarrow \fA$ be a surjective $*$-homomorphism and $\pi: \fA\rightarrow B(\H_\pi)$ be a $*$-representation. If $\pi\circ q$ possesses the approximate unique extension property with respect to $\M$, then $\pi$ possesses the approximate unique extension property with respect to $q(\M)$.
\end{lemma}

\begin{proof}
Suppose that $\psi: \fA\rightarrow B(\H_\pi)$ is a unital completely positive map such that $\psi\mid_{q(\M)} = \pi\mid_{q(\M)}$. Then $\varphi := \psi\circ q$ is a unital completely positive map satisfying $\varphi\mid_\M = (\pi\circ q)\mid_\M$. By the assumption, there is a net of unitary operators $u_\beta\in B(\H_\pi)$ such that \[ \lim_{\beta}u_\beta^*\varphi(t)u_\beta = (\pi\circ q)(t), \ \ \ \ \ \ t\in \rC^*(\M),\]in the weak operator topology. Thus, \[ \lim_{\beta}u_\beta^* \psi(s)u_\beta = \pi(s), \ \ \ \ \ \ s\in \fA,\]in the weak operator topology. So $\pi$ has the approximate unique extension property.
\end{proof}

As a completely isometric representation of $\M$ induces a $*$-isomorphism of $\rC^*$-envelopes \cite[pg. 219]{paulsen2002completely}, Lemma \ref{L:ApproxUEPquotient} implies that the $*$-representations of $\rC^*_e(\M)$ that possess the approximate unique extension property are invariant under completely isometric isomorphism.

\begin{proposition}\label{P:AUEPInv}
Let $\M\subset B(\H)$ and $\N\subset B(\K)$ be unital operator spaces that generate their respective $\rC^*$-envelopes. Suppose that $\iota:\M\rightarrow \N$ is a unital completely isometric linear map and let $\theta:\rC^*_e(\M)\rightarrow\rC^*_e(\N)$ be a $*$-isomorphism satisfying $\theta\mid_\M = \iota$. Then, $\pi:\rC^*_e(\N)\rightarrow B(\H_\pi)$ possesses the approximate unique extension property with respect to $\N$ if and only if $\pi\circ\theta$ possesses the approximate unique extension property with respect to $\M$.
\end{proposition}

\begin{proof}
If $\pi\circ\theta$ has the approximate unique extension property with respect to $\M$, then $\pi$ has the approximate unique extension property with respect to $\N$ by Lemma \ref{L:ApproxUEPquotient}.

Conversely, if $\pi = (\pi\circ\theta)\circ\theta^{-1}$ has the approximate unique extension property with respect to $\N$, then again $\pi\circ\theta$ has the approximate unique extension property with respect to $\M$ by Lemma \ref{L:ApproxUEPquotient}.
\end{proof}

Additionally, we remark that the approximate unique extension property has a characterization that is reminiscent of Proposition \ref{P:UEPasym}.

\begin{proposition}
Let $\M\subset B(\H)$ be a unital operator space and let $\pi:\rC^*(\M)\rightarrow B(\H_\pi)$ be a $*$-representation. Then, the following statements are equivalent.\begin{enumerate}[\rm (i)]
\item The $*$-representation $\pi$ possesses the approximate unique extension property.
\item If $\psi_\alpha:\rC^*(\M)\rightarrow B(\H_\pi)$ is a net of unital completely positive maps such that $\psi_\alpha(m)\rightarrow \pi(m)$ in the weak operator topology for every $m\in\M$, then there is a subnet $(\psi_{\alpha_\eta})_\eta$ of unital completely positive maps and a net of unitary operators $u_\beta\in B(\H_\pi)$ such that \[\lim_\beta\lim_\eta u_\beta^* \psi_{\alpha_\eta}(t)u_\beta = \pi(t), \ \ \ \ \ \ t\in\rC^*(\M),\]in the weak operator topology.
\end{enumerate}
\end{proposition}

\begin{proof}
{\rm (i)}$\Rightarrow${\rm (ii)}: Let $\psi_\alpha:\rC^*(\M)\rightarrow B(\H_\pi)$ be a net of unital completely positive maps such that $\psi_\alpha(m)\rightarrow \pi(m)$ in the weak operator topology for every $m\in\M$. By compactness, we may obtain a subnet $(\psi_{\alpha_\eta})_\eta$ of unital completely positive maps that converges to a unital completely positive map $\varphi:\rC^*(\M)\rightarrow B(\H_\pi)$ in the pointwise weak operator topology \cite[Theorem 7.4]{paulsen2002completely}. In other words, we have that $\psi_{\alpha_\eta}(t)\rightarrow\varphi(t)$ in the weak operator topology for every $t\in\rC^*(\M)$. Consequently, we have that $\varphi\mid_\M = \pi\mid_\M$.

Since $\pi$ has the approximate unique extension property, there is a net of unitary operators $u_\beta\in B(\H_\pi)$ such that $u_\beta^*\varphi(t)u_\beta\rightarrow\pi(t)$ in the weak operator topology for every $t\in\rC^*(\M)$. So we conclude that \[\lim_\beta\lim_\eta u_\beta^*\psi_{\alpha_\eta}(t)u_\beta = \lim_\beta u_\beta^*\varphi(t)u_\beta = \pi(t), \ \ \ \ \ \ t\in\rC^*(\M),\]as desired.

{\rm (ii)}$\Rightarrow${\rm (i)}: Trivial.
\end{proof}

We now work towards identifying a large class of $*$-representations that possess the approximate unique extension property. For this, we require that the approximate unique extension property be preserved under approximate unitary equivalence.

\begin{proposition}\label{P:AUEP_aue}
Let $\M\subset B(\H)$ be a unital operator space and let $\pi:\rC^*(\M)\rightarrow B(\H_\pi)$ be a $*$-representation. Suppose that $\pi$ possesses the approximate unique extension property. If $\sigma:\rC^*(\M)\rightarrow B(\H_\sigma)$ is a $*$-representation that is approximately unitarily equivalent to $\pi$, then $\sigma$ also possesses the approximate unique extension property.
\end{proposition}

\begin{proof}
The approximate unique extension property is easily seen to be stable under unitary equivalence. Thus, since there is a unitary operator $U:\H_\sigma\rightarrow\H_\pi$, we may assume for simplicity that $\H_\sigma = \H_\pi$. In this case, there is a net of unitary operators $w_\gamma\in B(\H_\pi)$ such that \[ \lim_\gamma \lVert w_\gamma^*\sigma(t)w_\gamma - \pi(t)\rVert=0, \ \ \ \ \ \ t\in \rC^*(\M).\]In particular, we also have that \[\lim_\gamma\| \sigma(t) - w_\gamma\pi(t) w_\gamma^*\| = 0, \ \ \ \ \ \ t\in\rC^*(\M).\]

Let $\psi: \rC^*(\M)\rightarrow B(\H_\pi)$ be a unital completely positive map such that $\psi\mid_\M = \sigma\mid_\M$. By compactness \cite[Theorem 7.4]{paulsen2002completely}, we may obtain a subnet $(w_{\gamma_\eta}^*\psi(\cdot)w_{\gamma_\eta})_{\eta}$ of unital completely positive maps and a unital completely positive map $\varphi:\rC^*(\M)\rightarrow B(\H_\pi)$ such that \[\lim_\eta w_{\gamma_\eta}^*\psi(t)w_{\gamma_\eta}=\varphi(t), \ \ \ \ \ \ t\in\rC^*(\M),\]in the weak operator topology. In particular, \[ \varphi(m) = \lim_{\eta}w_{\gamma_\eta}^*\psi(m)w_{\gamma_\eta} = \lim_\eta w_{\gamma_\eta}^*\sigma(m) w_{\gamma_\eta} = \pi(m), \ \ \ \ \ \ \ m\in\M.\]As $\pi$ has the approximate unique extension property, there is a net of unitaries $u_\beta\in B(\H_\pi)$ such that \[\pi(t) = \lim_{\beta} u_\beta^* \varphi(t) u_\beta = \lim_\beta \lim_\eta u_\beta^* w_{\gamma_\eta}^* \psi(t) w_{\gamma_\eta} u_\beta, \ \ \ \ \ \ t\in \rC^*(\M).\]Then we see that\[\lim_\gamma \lim_\beta \lim_\eta  w_\gamma u_\beta^* w_{\gamma_\eta}^* \psi(t) w_{\gamma_\eta} u_\beta w_\gamma^* = \lim_\gamma w_\gamma\pi(t)w_\gamma^* = \sigma(t), \ \ \ \ \ \ t\in \rC^*(\M).\]Thus, there is a net of unitaries $v_\alpha:=w_{\gamma_\alpha}u_{\beta_\alpha}^*w_{\gamma_{\eta_\alpha}}^*$ that corresponds to the above iterated limit (for example -- see \cite[pg. 69]{kelley2017general}) and so, we have that \[\lim_\alpha v_\alpha^*\psi(t)v_\alpha = \sigma(t), \ \ \ \ \ \ t\in\rC^*(\M).\]Therefore $\sigma$ has the approximate unique extension property.
\end{proof}

This allows us to obtain a particularly significant class of $*$-representations that possess the approximate unique extension property.

\begin{theorem}\label{T:UEPimplyAUEP}
Let $\M\subset B(\H)$ be a unital operator space. Suppose that $\pi:\rC^*(\M)\rightarrow B(\H_\pi)$ is a $*$-representation that is approximately unitarily equivalent to a $*$-representation with the unique extension property. Then, we have that $\pi$ has the approximate unique extension property.
\end{theorem}

\begin{proof}
By Proposition \ref{P:AUEP_aue}, it suffices to show that $*$-representations with the unique extension property have the approximate unique extension property. Thus, we assume that $\pi$ has the unique extension property. In which case, if $\psi: \rC^*(\M)\rightarrow B(\H_\pi)$ is a unital completely positive map satisfying $\psi\mid_\M = \pi\mid_\M$, then $\psi=\pi$ and so we may take $u_\beta$ to be the identity operator.
\end{proof}

We remark that one form of Voiculescu's Theorem states that, whenever $\pi$ and $\sigma$ are approximately unitarily equivalent $*$-representations of a separable $\rC^*$-algebra $\fA$ on a separable Hilbert space, then the limiting difference $\sigma(t)-u_n^*\pi(t)u_n$ can be made compact for every $t\in\fA$ and every $n\in\bN$ \cite[Theorem 1.5]{voiculescu1976non}. A similar statement has been obtained for tuples of separable operators \cite[Theorem 5.2]{davidson2017dilations}. Thus, one may wonder whether the difference in our setting, given by $\pi(t)-u_\beta^*\psi(t)u_\beta$, can be made compact under separability assumptions. However, we do not know if such a refinement is possible in our context due to the application of compactness in Proposition \ref{P:AUEP_aue}.

Since the usual unique extension property is not preserved under approximate unitary equivalence (Example \ref{E:CuntzInf}), by Theorem \ref{T:UEPimplyAUEP} we obtain that the approximate unique extension property is truly different than the usual unique extension property. Moreover, note that whenever $\pi:\rC^*(\M)\rightarrow B(\H_\pi)$ has the approximate unique extension property and $\psi:\rC^*(\M)\rightarrow B(\H_\pi)$ is a completely positive extension of $\pi\mid_\M$, we may conclude by \cite[Lemma 2.1]{kleski2014korovkin} that there is a net of unitaries $u_\beta\in B(\H_\pi)$ such that $u_\beta^*\psi(t)u_\beta\rightarrow \pi(t)$ in the \emph{strong operator topology} for every $t\in\rC^*(\M)$. In particular, one may not interchange the roles of the maps $\psi$ and $\pi$ in Proposition \ref{P:UEPbackwardsApprox}.

Notwithstanding the above differences, the approximate unique extension property still behaves well enough to form a well-behaved class of $*$-representations.

\begin{proposition}\label{P:AUEPSum}
Let $\M\subset B(\H)$ be a unital operator space. If $\{\pi_i:\rC^*(\M)\rightarrow B(\H_i) : i\in I\}$ is a collection of $*$-representations that possess the approximate unique extension property, then $\pi = \bigoplus_{i\in I}\pi_i$ also possesses the approximate unique extension property.
\end{proposition}

\begin{proof}
Let $\psi:\rC^*(\M)\rightarrow B(\bigoplus_{i\in I}\H_i)$ be a unital completely positive map satisfying $\psi\mid_\M = \pi\mid_\M$. For each $i$, define a unital completely positive map \[\psi_i:\rC^*(\M)\rightarrow B(\H_i), \ \ \ \ \ \ t\mapsto P_{\H_i}\psi(t)\mid_{\H_i},\] and note that $\psi_i\mid_\M= \pi_i\mid_\M$. In turn, as each $\pi_i$ possesses the approximate unique extension property, we may obtain a directed set $\Lambda_i$ and a corresponding net of unitary operators $u_{\lambda_i}\in B(\H_i)$, $\lambda_i\in\Lambda_i,$ such that \[\lim_{\lambda_i\in\Lambda_i}u_{\lambda_i}^*\psi_i(t)u_{\lambda_i}= \pi_i(t), \ \ \ \ \ \ t\in\rC^*(\M),\]where convergence is in the weak operator topology.

Now, let $\F$ denote the directed set consisting of finite subsets of $I$. Given $F = \{ i_1, \ldots, i_n\}\in\F$ and $\lambda_{i_k}\in\Lambda_{i_k}$, define a unitary operator on $\bigoplus_{i\in I}\H_i$ by \[u_{F, \lambda_{i_1}, \ldots, \lambda_{i_n}} = \bigoplus_{i\in \F}u_{\lambda_i}\oplus\bigoplus_{i\notin\F} I_{\H_i}\]where $I_{\H_i}$ denotes the identity operator on $\H_i$. For $t\in\rC^*(\M)$, we have that\[ P_{\H_i}u_{F,\lambda_{i_1}, \ldots, \lambda_{i_n}}^* \psi(t) u_{F,\lambda_{i_1}, \ldots, \lambda_{i_n}}\mid_{\H_i} = \begin{cases}u_{\lambda_i}^*\psi_i(t)u_{\lambda_i}, & i\in F\\ \psi_i(t), & i\notin F.\end{cases}\]Thus, we must have that the iterated limit \[\lim_{\lambda_{i_1}\in\Lambda_{i_1}}\lim_{\lambda_{i_2}\in\Lambda_{i_2}}\ldots\lim_{\lambda_{i_n}\in\Lambda_{i_n}} P_{\H_i}u_{F,\lambda_{i_1}, \ldots, \lambda_{i_n}}^* \psi(t) u_{F,\lambda_{i_1}, \ldots, \lambda_{i_n}}\mid_{\H_i}\]exists in the weak operator topology for any $i\in I$ and any $F=\{i_1, \ldots, i_n\}\in\F$. Indeed, if $i\in\F$, then the limit equals $\pi_i(t)$ and, when $i\notin\F$, we have that the limit is equal to $\psi_i(t)$. So\[\lim_{\{i_1,\ldots,i_n\}\in\F}\lim_{\lambda_{i_1}\in\Lambda_{i_1}}\lim_{\lambda_{i_2}\in\Lambda_{i_2}}\ldots\lim_{\lambda_{i_n}\in\Lambda_{i_n}} P_{\H_i}u_{F,\lambda_{i_1}, \ldots, \lambda_{i_n}}^* \psi(t) u_{F,\lambda_{i_1}, \ldots, \lambda_{i_n}}\mid_{\H_i} = \pi_i(t)\]for every $i\in I$. Thus, by \cite[pg. 69]{kelley2017general}, there is a net of unitary operators $u_\beta\in B(\bigoplus_{i\in I}\H_i)$ such that \[ \lim_\beta P_{\H_i} u_\beta^* \psi(t) u_\beta\mid_{\H_i} = \pi_i(t), \ \ \ \ \ \ t\in\rC^*(\M),\]for every $i\in I$.

By compactness, we may extract a subnet $(u_{\beta_\eta}^*\psi(\cdot)u_{\beta_\eta})_\eta$ and a unital completely positive map $\varphi:\rC^*(\M)\rightarrow B(\bigoplus_{i\in I}\H_i)$ such that \[\lim_\eta u_{\beta_\eta}^*\psi(t)u_{\beta_\eta}=\varphi(t), \ \ \ \ \ \ \ t\in\rC^*(\M),\]in the weak operator topology \cite[Theorem 7.4]{paulsen2002completely}. From here, the argument is a direct adaptation of \cite[Proposition 4.4]{arveson2011noncommutative}. Indeed, observe that $P_{\H_i}\varphi(t)P_{\H_i} = \pi(t)P_{\H_i}$ for each $i\in I$ and so, by the Schwarz inequality, \begin{align*}
P_{\H_i}\varphi(t)^*(I- P_{\H_i})\varphi(t) P_{\H_i} & = P_{\H_i}\varphi(t)^*\varphi(t)P_{\H_i} - P_{\H_i} \varphi(t)^* P_{\H_i} \varphi(t)P_{\H_i}\\
& \leq P_{\H_i} \varphi(t^*t)P_{\H_i}  - P_{\H_i}\varphi(t)^* P_{\H_i} \varphi(t) P_{\H_i}\\
& = \pi(t^*t)P_{\H_i} - \pi(t)^*\pi(t)P_{\H_i}=0.
\end{align*}Therefore $(I-P_{\H_i})\varphi(t)P_{\H_i} = 0$ and so \[ \lim_\eta u_{\beta_\eta}^*\psi(t)u_{\beta_\eta} = \varphi(t)=\sum_{i\in I}\varphi(t) P_{\H_i} = \sum_{i\in I} P_{\H_i}\varphi(t)P_{\H_i} = \sum_{i\in I}\pi(t)P_{\H_i} = \pi(t).\]
\end{proof}

Recall that Lemma \ref{L:UEP} states that the implication in Proposition \ref{P:AUEPSum} is known to be an equivalence when the approximate unique extension property is replaced by the usual unique extension property. On the contrary, the converse will be shown to be far from true for $*$-representations that possess the approximate unique extension property (Proposition \ref{P:AUEPsubrep}). First, we require a consequence to Proposition \ref{P:AUEPSum} that allows us to construct many more $*$-representations that possess the approximate unique extension property.

\begin{proposition}\label{P:AUEPker}
Let $\M\subset B(\H)$ be a unital operator space. Suppose that $\pi:\rC^*(\M)\rightarrow B(\H_\pi)$ is a $*$-representation possessing the approximate unique extension property. Then, the following statements hold.\begin{enumerate}[\rm (i)]
\item If $\sigma:\rC^*(\M)\rightarrow B(\H_\sigma)$ is a $*$-representation such that $\ker\sigma = \ker\pi$, then there is some cardinal $\kappa$ such that $\sigma^{(\kappa)}$ possesses the approximate unique extension property.
\item If $\rho: \rC^*_e(\M)\rightarrow B(\H_\rho)$ is an isometric $*$-representation, then there is a cardinal $\mu$ such that $\rho^{(\mu)}$ possesses the approximate unique extension property.
\end{enumerate}
\end{proposition}

\begin{proof}
{\rm (i)}: Define an infinite cardinal $\kappa = \max\{\dim\H_\pi,\dim\H_\sigma, \aleph_0\}$. For each $t\in\rC^*(\M)\setminus \ker\sigma$, we have that \[\text{rank}(\sigma^{(\kappa)}(t)) =\kappa\text{rank}(\sigma(t)) = \kappa,\]and likewise, $\text{rank}(\pi^{(\kappa)}(t)) = \kappa.$ Since $\ker\pi=\ker\sigma$, Theorem \ref{T:HadwinAUE} implies that $\sigma^{(\kappa)}$ is approximately unitarily equivalent to $\pi^{(\kappa)}$. By Propositions \ref{P:AUEP_aue} and \ref{P:AUEPSum}, we find that $\sigma^{(\kappa)}$ possesses the approximate unique extension property.

{\rm (ii)}: By Theorem \ref{T:HR} {\rm (i)}, there is an isometric $*$-representation $\beta:\rC^*_e(\M)\rightarrow B(\H_\beta)$ that possesses the unique extension property. Thus, it follows by {\rm (i)} that there is a cardinal $\mu$ such that $\rho^{(\mu)}$ possesses the approximate unique extension property.
\end{proof}

Due to Proposition \ref{P:AUEPker} {\rm (ii)}, we may confirm that the approximate unique extension property is badly behaved with respect to sub-representations.

\begin{proposition}\label{P:AUEPsubrep}
Let $\M\subset B(\H)$ be a unital operator space. Assume that whenever $\pi:\rC^*_e(\M)\rightarrow B(\H_\pi)$ is a $*$-representation possessing the approximate unique extension property and $\F\subset\H_\pi$ is a closed subspace that is reducing for $\text{Im}\pi$, then we have that the $*$-representation\[t\mapsto\pi(t)\mid_\F, \ \ \ \ \ \ \ t\in\rC^*_e(\M),\]also possesses the approximate unique extension property. Then, we have that every $*$-representation of $\rC^*_e(\M)$ has the approximate unique extension property.
\end{proposition}

\begin{proof}
Let $\beta$ be an isometric $*$-representation of $\rC^*_e(\M)$ on a separable Hilbert space that possesses the approximate unique extension property (which exists by either Theorem \ref{T:HR} or Proposition \ref{P:AUEPker}) and let $\pi:\rC^*_e(\M)\rightarrow B(\H_\pi)$ be a $*$-representation. Then, there is a cardinal $\kappa$ such that $\beta^{(\kappa)}$ has the approximate unique extension property by Proposition \ref{P:AUEPSum}. By Theorem \ref{T:HadwinAUE}, we have that $\beta^{(\kappa)}\oplus\pi$ is approximately unitarily equivalent to $\beta^{(\kappa)}$. So $\beta^{(\kappa)}\oplus\pi$ possesses the approximate unique extension property by Propostion \ref{P:AUEP_aue}. By the assumption, we conclude that $\pi$ has the approximate unique extension property. 
\end{proof}

We conclude this section by connecting our work with Arveson's hyperrigidity conjecture. To start, we note the following.

\begin{proposition}\label{P:AUEPHR}
Let $\M\subset B(\H)$ be a separable unital operator space. Then, the following statements are equivalent.\begin{enumerate}[\rm (i)]
\item Every irreducible $*$-representation of $\rC^*(\M)$ possesses the approximate unique extension property.
\item Every $*$-representation of $\rC^*(\M)$ possesses the approximate unique extension property.
\end{enumerate}
\end{proposition}

\begin{proof}
For the non-trivial direction, assume that every irreducible $*$-representation of $\rC^*(\M)$ possesses the approximate unique extension property and let $\pi:\rC^*(\M)\rightarrow B(\H_\pi)$ be a $*$-representation. As $\rC^*(\M)$ is separable, we may express $\pi = \bigoplus_{i\in I}\pi_i$ where $\pi_i:\rC^*(\M)\rightarrow B(\H_i)$ is a $*$-representation on a separable Hilbert space $\H_i$. Since $\H_i$ is separable, we have that $\pi_i$ is approximately unitarily equivalent to a direct sum of irreducible $*$-representations \cite[Corollary 1.6]{voiculescu1976non}. By the assumption, Proposition \ref{P:AUEP_aue}, and Proposition \ref{P:AUEPSum}, we may conclude that each $\pi_i$ has the approximate unique extension property. Thus, $\pi$ has the approximate unique extension property by Proposition \ref{P:AUEPSum}.
\end{proof}

We remain unsure if separability can be removed in Proposition \ref{P:AUEPHR}. Indeed, we required work of Voiculescu \cite{voiculescu1976non}, which states that every $*$-representation is approximately unitarily equivalent to a direct sum of irreducible $*$-representations provided that certain separability conditions are met. However, we are unaware if there is an analogue to this fact in the non-separable setting.

Also, note that Proposition \ref{P:AUEPHR} may be viewed as a positive answer to an approximate version of Arveson's hyperrigidity conjecture. Indeed, as reflected in Proposition \ref{P:AUEPHR}, to determine the class of $*$-representations that possess the approximate unique extension property, it is sufficient to study those \emph{irreducible} $*$-representations that possess the approximate unique extension property.

We may now state the main result of this section.

\begin{theorem}\label{T:AUESummary}
Let $\M\subset B(\H)$ be a separable unital operator space. Consider the following statements:\begin{enumerate}[\rm (i)]
\item The operator space $\M$ is hyperrigid.
\item Every irreducible $*$-representation of $\rC^*(\M)$ is a boundary representation.
\item Every irreducible $*$-representation of $\rC^*(\M)$ is approximately unitarily equivalent to a boundary representation.
\item Every $*$-representation of $\rC^*(\M)$ has the approximate unique extension property.
\end{enumerate}Then, we have that {\rm (i)}$\Rightarrow${\rm (ii)}$\Rightarrow${\rm (iii)}$\Rightarrow${\rm (iv)} and {\rm (iii)}$\not\Rightarrow${\rm (ii)}.\end{theorem}

\begin{proof}
The directions {\rm (i)}$\Rightarrow${\rm (ii)}$\Rightarrow${\rm (iii)} are trivial and {\rm (iii)}$\Rightarrow${\rm (iv)} is a consequence to Theorem \ref{T:UEPimplyAUEP} and Proposition \ref{P:AUEPHR}. Finally, {\rm (iii)}$\not\Rightarrow${\rm (ii)} by considering Example \ref{E:CuntzInf}.
\end{proof}

Later, it will be revealed that the approximate unique extension property coincides with the usual unique extension property for those irreducible $*$-representations whose image contains a compact operator (Theorem \ref{T:Compact}). Consequently, statements {\rm (ii)}, {\rm (iii)}, and {\rm (iv)} will be equivalent whenever $\rC^*(\M)$ is postliminal.

\section{Split sequences for completely positive extensions}\label{S:SplitSeq}

Suppose that $\pi$ has the approximate unique extension property with respect to a unital operator space $\M$. We will now present a restriction on the structure of the unital completely positive extensions of $\pi\mid_\M$. This will be the driving force of our arguments throughout this section and is the main result of our work.

For this recall that, given $\rC^*$-algebras $\fA\subset\fB$, a \emph{conditional expectation} is a contractive linear projection $E:\fB\rightarrow\fA$. When $E$ is multiplicative, we say that the map is \emph{homomorphic}. By Tomiyama's Theorem \cite[Theorem 1.5.10]{brown2008textrm}, a conditional expectation is automatically completely positive and satisfies $E(aba') = a E(b)a'$ for every $a,a'\in\fA$ and $b\in\fB$. Consequently, a homomorphic conditional expectation is a $*$-homomorphism.

\begin{theorem}\label{T:AUEPsequence}
Let $\M\subset B(\H)$ be a unital operator space. Suppose that $\pi: \rC^*(\M)\rightarrow B(\H_\pi)$ is a $*$-representation possessing the approximate unique extension property and let $\psi: \rC^*(\M)\rightarrow B(\H_\pi)$ be a unital completely positive map satisfying $\psi\mid_\M = \pi\mid_\M$. Then, the following statements hold.\begin{enumerate}[\rm (i)]
\item There is a homomorphic conditional expectation $\Gamma_\psi: \rC^*(\text{Im}\psi)\rightarrow \text{Im}\pi$ satisfying $\Gamma_\psi\circ\psi = \pi$.
\item If $\H_\pi$ is separable, then $\text{rank}(\pi(t))\leq \text{rank}(\psi(t))$ for each $t\in \rC^*(\M)$.
\item We have that $\ker\Gamma_\psi$ is the smallest closed two-sided ideal of $\rC^*(\text{Im}\psi)$ that contains $\text{Im}(\psi-\pi)$.
\item We have that $\text{Im}(id-\Gamma_\psi) = \ker\Gamma_\psi$ and $\rC^*(\text{Im}\psi) = \ker\Gamma_\psi+\text{Im}\pi$.
\item We have that $\psi = \pi$ if and only if $\Gamma_\psi$ is isometric.
\end{enumerate}
\end{theorem}

\begin{proof}
{\rm (i)}: As $\pi$ possesses the approximate unique extension property, there is a net of unitaries $u_\beta\in B(\H_\pi)$ such that $u_\beta^*\psi(t)u_\beta\rightarrow \pi(t)$ in the weak operator topology for each $t\in \rC^*(\M)$. By \cite[Lemma 2.1]{kleski2014korovkin}, we may conclude that convergence is in the strong operator topology. Consequently, given a $*$-polynomial $p$ in non-commuting variables and finitely many elements $t_1, \ldots, t_n\in \rC^*(\M)$, we have that \[\lim_\beta u_\beta^* p(\psi(t_1), \ldots, \psi(t_n)) u_\beta = p(\pi(t_1), \ldots, \pi(t_n))\]in the strong operator topology. Since the linear span of words in $\text{Im}\psi$ is dense within $\rC^*(\text{Im}\psi)$, we obtain a $*$-homomorphism \[\Gamma_\psi: \rC^*(\text{Im}\psi)\rightarrow\text{Im}\pi,  \ \ \ \ \ \ x\mapsto \lim_\beta u_\beta^*xu_\beta.\]Evidently, $\Gamma_\psi\circ\psi = \pi$. Note that \[\text{Im}\pi= \rC^*(\pi(\M)) = \rC^*(\psi(\M)) \subset\rC^*(\text{Im}\psi).\] Since \[(\Gamma_\psi\circ\pi)\mid_\M = (\Gamma_\psi\circ\psi)\mid_\M = \pi\mid_\M\]and $\Gamma_\psi$ is a $*$-representation, it follows that $\Gamma_\psi\circ\pi = \pi$. Therefore, $\Gamma_\psi$ is a conditional expectation as desired.

{\rm (ii)}: It is well-known that the rank function is lower semicontinous in the weak operator topology provided that $\H_\pi$ is separable. Then {\rm (ii)} follows instantly since there is a net of unitaries $u_\beta\in B(\H_\pi)$ such that $u_\beta^*\psi(t)u_\beta\rightarrow \pi(t)$ in the weak operator topology.

{\rm (iii)}: Since $\Gamma_\psi\circ\psi = \pi=\Gamma_\psi\circ\pi$, it follows that \[\text{Im}(\psi-\pi)\subset\ker\Gamma_\psi.\]Let $\fJ$ be the smallest closed two-sided ideal of $\rC^*(\text{Im}\psi)$ containing $\text{Im}(\psi-\pi)$. Necessarily, we have that $\fJ\subset\ker\Gamma_\psi$. We will show that each element $\zeta\in\rC^*(\text{Im}\psi)$ may expressed as $J+\pi(t)$ where $J\in\fJ$ and $t\in\rC^*(\M)$. Observe that \[\fD = \{ J+\pi(t) ~:~ J\in\fJ,\ t\in\rC^*(\M)\}\]is a closed subset of $\rC^*(\text{Im}\psi)$ as it is the pre-image of $\text{Im}\pi/\fJ$ under the quotient map $q:\rC^*(\text{Im}\psi)\rightarrow\rC^*(\text{Im}\psi)/\fJ$. Hence, by density, it suffices to consider when $\zeta\in\rC^*(\text{Im}\psi)$ lies within the linear span of words in $\text{Im}\psi$. In this case, since $\psi$ and $\rC^*(\M)$ are unital, we may obtain a finite sum of words\[ \zeta = \sum_i\psi(t_{i_1})\ldots\psi(t_{i_n})\]where $n$ is a fixed integer and $t_{k_\ell}\in\rC^*(\M)$ for each $k,\ell$. For a fixed $i$, we will show that $\psi(t_{i_1})\ldots\psi(t_{i_n})$ may be expressed as $J+\pi(t)$ by an easy inductive argument. For the base step, observe that \[\psi(t_{i_1}) = \psi(t_{i_1})-\pi(t_{i_1})+\pi(t_{i_1})\]and $(\psi(t_{i_1})-\pi(t_{i_1}))\in\fJ$. Now, for the inductive step, we may obtain $K\in\fJ$ and $s\in\rC^*(\M)$ such that $\psi(t_{i_1})\ldots\psi(t_{i_{n-1}}) = K+\pi(s)$. Then we see that \begin{align*}
\psi(t_{i_1})\ldots\psi(t_{i_n}) & = (K+\pi(s))\psi(t_{i_n})\\
& = (K+\pi(s))(\psi(t_{i_n})-\pi(t_{i_n})+\pi(t_{i_n}))\\
& = K(\psi(t_{i_n})-\pi(t_{i_n})+\pi(t_{i_n})) + \pi(s)(\psi(t_{i_n})-\pi(t_{i_n})) + \pi(st_{i_n}).
\end{align*}Therefore, we have that the claim also holds for $\zeta$.

Now, let $k\in\ker\Gamma_\psi$ and express $k = J+\pi(t)$ where $J\in\fJ$ and $t\in\rC^*(\M)$. Then we obtain that \[ k-J = \pi(t)= \Gamma_\psi(\pi(t)) = \Gamma_\psi(k-J) = 0.\]So $k = J$ and we have that $\fJ = \ker\Gamma_\psi$.

{\rm (iv):} This follows from a standard fact on linear idempotents, but we record it for later arguments. Since $\Gamma_\psi$ is a linear idempotent, it is easy to see that $\text{Im}(id-\Gamma_\psi) = \ker\Gamma_\psi$. Then, \[\rC^*(\text{Im}\psi) = \text{Im}(id-\Gamma_\psi) + \text{Im}\Gamma_\psi = \ker\Gamma_\psi + \text{Im}\pi.\] 

{\rm (v):} This is an immediate consequence of {\rm (iii)}.
\end{proof}

A statement similar to (i) was already known for faithful $*$-representations of nuclear $\rC^*$-algebras satisfying the property that all factor representations possess the unique extension property \cite[Theorem 4.2]{kleski2014korovkin}. The construction therein required non-trivial machinery on direct integrals and disintegration of measures. In contrast, our methods in Theorem \ref{T:AUEPsequence} provide a non-trivial implication for all unital operator spaces that fail to be hyperrigid. Indeed, whenever $\M\subset\rC^*_e(\M)$ fails to be hyperrigid, there is a $*$-representation of $\rC^*_e(\M)$ that has the approximate unique extension property but not the usual unique extension property by Proposition \ref{P:hr_aue} and Theorem \ref{T:UEPimplyAUEP}.

As a consequence to Theorem \ref{T:AUEPsequence}, we may derive a statement of independent interest. As far as we are aware, the following has not been recorded elsewhere.

\begin{corollary}\label{C:AUEPisom}
Let $\M\subset B(\H)$ be a unital operator space. Suppose that $\pi:\rC^*_e(\M)\rightarrow B(\H_\pi)$ is an isometric $*$-representation and let $\psi:\rC^*_e(\M)\rightarrow B(\H_\pi)$ be a unital completely positive map satisfying $\psi\mid_\M = \pi\mid_\M$. Then, the following statements hold.\begin{enumerate}[\rm (i)]
\item The map $\psi$ is completely isometric.
\item We have that $\rC^*_e(\text{Im}\psi)\cong \rC^*_e(\M)$.
\end{enumerate}
\end{corollary}

\begin{proof}
{\rm (i)}: By Proposition \ref{P:AUEPker}, there is a cardinal $\kappa$ such that $\Pi := \pi^{(\kappa)}$ has the approximate unique extension property. Note that the map $\Psi := \psi^{(\kappa)}$ agrees with $\Pi$ over $\M$. Thus, by Theorem \ref{T:AUEPsequence}, there is a homomorphic conditional expectation $\Gamma_\Psi: \rC^*(\text{Im}\Psi)\rightarrow\text{Im}\Pi$ satisfying $\Gamma_{\Psi}\circ\Psi = \Pi$. Since $\Pi$ is isometric, it immediately follows that $\Psi$ is completely isometric. Finally, this implies that $\psi$ itself must be completely isometric.

{\rm (ii)}: By statement {\rm (i)}, there is a complete order isomorphism $\iota: \text{Im}\psi\rightarrow \text{Im}\pi$ such that $\iota\circ\psi = \pi$. Since $\text{Im}\pi$ is a $\rC^*$-algebra, it then follows that $\text{Im}\pi\cong \rC^*_e(\M)$ is the $\rC^*$-envelope for $\text{Im}\psi$.
\end{proof}

The argument in Corollary \ref{C:AUEPisom} may be repurposed. If $\pi$ has the approximate unique extension property and $\psi$ is a completely positive extension such that $\ker\pi\subset\ker\psi$, then in fact we have equality of kernels and hence, there is a complete order isomorphism $\iota:\text{Im}\psi\rightarrow\text{Im}\pi$ given by $\Gamma_\psi\mid_{\text{Im}\psi}$. In the language of \cite[Theorem 10.5]{kennedy2021noncommutative}, we may conclude that the nc state space of the operator system $\text{Im}\psi$ is an ``nc Bauer simplex."

We briefly remark that the existence of the conditional expectation $\Gamma_\psi$ provides a tool for determining when a completely positive extension $\psi$ of $\pi\mid_\M$ will coincide with the $*$-representation $\pi$.

\begin{corollary}\label{C:OrthComp}
Let $\M\subset B(\H)$ be a unital operator space. Suppose that $\pi: \rC^*(\M)\rightarrow B(\H_\pi)$ is a $*$-representation possessing the approximate unique extension property and let $\psi: \rC^*(\M)\rightarrow B(\H_\pi)$ be a unital completely positive map satisfying $\psi\mid_\M = \pi\mid_\M$. Then, the following statements are equivalent.\begin{enumerate}[\rm (i)]
\item The equality $\psi = \pi$ holds.
\item The linear map $(id-\Gamma_\psi)$ is contractive.
\end{enumerate}
\end{corollary}

\begin{proof}
It is clear from Theorem \ref{T:AUEPsequence} that $\Gamma_\pi$ is the identity representation and so, {\rm (i)}$\Rightarrow$ {\rm (ii)} is immediate.

{\rm (ii)}$\Rightarrow${\rm (i)}: The map $(id-\Gamma_\psi)$ is a linear projection onto $\ker\Gamma_\psi$ by Theorem \ref{T:AUEPsequence} {\rm (iv)}. By the assumption, combined with Tomiyama's Theorem \cite[Theorem 1.5.10]{brown2008textrm}, we conclude that $(id-\Gamma_\psi)$ is a conditional expectation. Thus, given $k\in\ker\Gamma_\psi$ and $x\in\rC^*(\text{Im}\psi)$, we have that \[ xk = (id-\Gamma_\psi)(xk) = (id-\Gamma_\psi)(x)k = xk-\Gamma_\psi(x)k\]and so $\Gamma_\psi(x)k = 0$. As $x\in\rC^*(\text{Im}\psi)$ is arbitrary, we must have that $k=0$. So $\Gamma_\psi$ is injective and thus, $\psi = \pi$ by Theorem \ref{T:AUEPsequence} {\rm (v)}.
\end{proof}

For our next development, we will require some terminology. Let $\M\subset \N$ be unital operator spaces and let $\rU\rC\rP(\M,\N)$ denote the collection of unital completely positive maps from $\M$ to $\N$. There is a natural ordering among the idempotent elements of $\rU\rC\rP(\M,\N)$ defined by $\varphi\prec\varphi'$ if and only if $\varphi\circ\varphi' = \varphi'\circ\varphi = \varphi$. Given a subspace $\F$ of $\M$, an idempotent $\varphi$ is said to be a \emph{minimal $\F$-projection} in $\rU\rC\rP(\M,\N)$ if $\varphi$ is minimal among the idempotent elements of $\rU\rC\rP(\M,\N)$ that fix $\F$. Minimal $\F$-projections have found previous success in helping to prove existence of the injective envelope of an operator system \cite[Chapter 15]{paulsen2002completely}. We will find that the map $\Gamma_\psi$ enjoys this minimality condition and is uniquely determined by its behaviour over $\text{Im}\psi$.

\begin{theorem}\label{T:GammaUEP}
Let $\M\subset B(\H)$ be a unital operator space. Suppose that $\pi: \rC^*(\M)\rightarrow B(\H_\pi)$ is a $*$-representation possessing the approximate unique extension property and let $\psi: \rC^*(\M)\rightarrow B(\H_\pi)$ be a unital completely positive map satisfying $\psi\mid_\M = \pi\mid_\M$. Then, the following statements hold.\begin{enumerate}[\rm (i)]
\item Every $*$-representation $\sigma:\rC^*(\text{Im}\psi)\rightarrow B(\H_\sigma)$ such that $\sigma\mid_{\ker\Gamma_\psi}=0$ has the unique extension property with respect to $\text{Im}\psi$.
\item The map $\Gamma_\psi$ is a minimal $\text{Im}\pi$-projection in $\rU\rC\rP(\rC^*(\text{Im}\psi), B(\H_\pi))$.
\end{enumerate}
\end{theorem}

\begin{proof}
{\rm (i)}: Let $\varphi:\rC^*(\text{Im}\psi)\rightarrow B(\H_\sigma)$ be a unital completely positive map satisfying $\varphi\mid_{\text{Im}\psi} =\sigma\mid_{\text{Im}\psi}.$ As $\sigma$ vanishes over $\ker\Gamma_\psi$, we may express $\sigma = \widetilde{\sigma}\circ\Gamma_\psi$ where $\widetilde{\sigma}:\text{Im}\pi\rightarrow B(\H_\sigma)$ is a $*$-representation. Observe that, by Theorem \ref{T:AUEPsequence}, \[ (\varphi\circ\psi)(t) = (\sigma\circ\psi)(t) = (\widetilde{\sigma}\circ\Gamma_\psi\circ\psi)(t) = (\widetilde{\sigma}\circ\pi)(t), \ \ \ \ \ t\in\rC^*(\M).\]By the Schwarz inequality, we have that for each $t\in\rC^*(\M)$, \[\varphi(\psi(t)^*\psi(t))\leq (\varphi\circ\psi)(t^*t) = (\widetilde{\sigma}\circ\pi)(t^*t) = \varphi(\psi(t))^*\varphi(\psi(t)) \leq \varphi(\psi(t)^*\psi(t))\]and so $\varphi(\psi(t)^*\psi(t)) = \varphi(\psi(t))^*\varphi(\psi(t))$. Consequently, we have that $\text{Im}\psi$ is a subset of the multiplicative domain for $\varphi$ \cite[Theorem 3.18]{paulsen2002completely}. Therefore, since $\varphi\mid_{\text{Im}\psi} = \sigma\mid_{\text{Im}\psi}$, we must have that $\varphi=\sigma$.

{\rm (ii)}: Let $\Omega:\rC^*(\text{Im}\psi)\rightarrow B(\H_\pi)$ be a unital completely positive idempotent such that $\Omega\circ\pi = \pi$ and $\Omega\prec\Gamma_\psi$. In this case, by Theorem \ref{T:AUEPsequence}, \[ \Omega\circ\psi = \Omega\circ\Gamma_\psi\circ\psi = \Omega\circ\pi = \pi.\]As $\Omega\mid_{\text{Im}\psi}=\Gamma_\psi\mid_{\text{Im}\psi}$, it follows by {\rm (i)} that $\Omega = \Gamma_\psi$.
\end{proof}

As a result of Theorem \ref{T:GammaUEP}, we show next that there is a one-to-one correspondence between the idempotents $\Gamma_\psi$ and the domain of such maps. In particular, a completely positive extension $\psi$ will be equal to a $*$-representation $\pi$ with the approximate unique extension property provided that the image of the completely positive extension is sufficiently small. This generalizes \cite[Corollary 3.3]{clouatre2018unperforated}.

\begin{corollary}\label{C:Correspondence}
Let $\M\subset B(\H)$ be a unital operator space. Suppose that $\pi:\rC^*(\M)\rightarrow B(\H_\pi)$ is a $*$-representation possessing the approximate unique extension property and let $\psi,\psi':\rC^*(\M)\rightarrow B(\H_\pi)$ be unital completely positive maps satisfying $\psi\mid_\M = \psi'\mid_\M = \pi\mid_\M$. Then, the following statements are equivalent.\begin{enumerate}[\rm (i)]
\item We have that $\ker\Gamma_\psi \subset\ker\Gamma_{\psi'}$.
\item We have that $\rC^*(\text{Im}\psi) \subset \rC^*(\text{Im}\psi')$.
\item We have that $\Gamma_{\psi'}\mid_{\rC^*(\text{Im}\psi)} = \Gamma_\psi$.
\end{enumerate}In particular, we have that $\Gamma_\psi = \Gamma_{\psi'}$ if and only if $\rC^*(\text{Im}\psi) = \rC^*(\text{Im}\psi').$ Further, $\psi=\pi$ whenever $\text{Im}\psi\subset\text{Im}\pi$.
\end{corollary}

\begin{proof}
The implication {\rm (iii)}$\Rightarrow${\rm (i)} is trivial and {\rm (i)}$\Rightarrow${\rm (ii)} follows directly from Theorem \ref{T:AUEPsequence} {\rm (iv)} because\[ \rC^*(\text{Im}\psi) = \ker\Gamma_\psi + \text{Im}\pi \subset \ker\Gamma_{\psi'}+\text{Im}\pi = \rC^*(\text{Im}\psi').\]

{\rm (ii)}$\Rightarrow${\rm (iii)}: In this case, we obtain a well-defined unital completely positive map $\varphi := \Gamma_{\psi'}\circ\psi: \rC^*(\M)\rightarrow\text{Im}\pi$ satisfying \[\varphi\mid_\M = (\Gamma_{\psi'}\circ\psi)\mid_\M =(\Gamma_{\psi'}\circ\pi)\mid_\M = \pi\mid_\M.\]Thus, by Theorem \ref{T:AUEPsequence} {\rm (i)}, \[ \pi(t) = (\Gamma_\varphi\circ\varphi)(t) = \varphi(t), \ \ \ \ \ \ \ t\in\rC^*(\M),\]on account of $\varphi(t)$ lying in the image of the idempotent $\Gamma_\varphi: \rC^*(\text{Im}\varphi)\rightarrow\text{Im}\pi.$ It follows that $\Gamma_{\psi'}\circ\psi = \pi$. However, by Theorem \ref{T:GammaUEP} {\rm (i)}, the map $\Gamma_\psi$ has the unique extension property with respect to $\text{Im}\psi$. So, since $\Gamma_{\psi'}\mid_{\text{Im}\psi} = \Gamma_\psi\mid_{\text{Im}\psi}$, it follows that $\Gamma_{\psi'}\mid_{\rC^*(\text{Im}\psi)} = \Gamma_\psi$.

For the final statement note that, whenever $\text{Im}\psi\subset\text{Im}\pi$, it follows from condition {\rm (iii)} that \[\Gamma_\psi = \Gamma_\pi\mid_{\rC^*(\text{Im}\psi)} = id,\]where $id$ denotes the identity representation of $\rC^*(\text{Im}\psi)$. By Theorem \ref{T:AUEPsequence} {\rm (v)}, we conclude that $\psi = \pi$.
\end{proof}

Recall that, by Proposition \ref{P:AUEPHR}, to determine the class of $*$-representations that possess the approximate unique extension property it suffices to determine the class of \emph{irreducible} $*$-representations that possess the approximate unique extension property (provided that $\M$ is separable). Any other $*$-representation is the result of a direct sum or an approximate unitary equivalence of $*$-representations of this form. Here, we will mention how the conditional expectation $\Gamma_\psi$ behaves with respect to direct sums and approximate unitary equivalences of completely positive extensions. We start with direct sums.

\begin{corollary}\label{C:GammaSum}
Let $\M\subset B(\H)$ be a unital operator space and $\{\pi_i:\rC^*(\M)\rightarrow B(\H_i) : i\in I\}$ be a collection of $*$-representations that possess the approximate unique extension property. For each $i\in I$, let $\psi_i:\rC^*(\M)\rightarrow B(\H_i)$ be a unital completely positive map satisfying $\psi_i\mid_\M = \pi_i\mid_\M$. Then, defining $\pi = \bigoplus_{i\in I}\pi_i$ and $\psi = \bigoplus_{i\in I}\psi_i$, we have that the map $\Gamma_\psi:\rC^*(\text{Im}\psi)\rightarrow\text{Im}\pi$ satisfies\[\Gamma_\psi((x_i)_{i\in I}) = (\Gamma_{\psi_i}(x_i))_{i\in I}, \ \ \ \ \ \ (x_i)_{i\in I}\in\rC^*(\text{Im}\psi).\]
\end{corollary}

\begin{proof}
Note that $\psi$ is a unital completely positive map satisfying $\psi\mid_\M = \pi\mid_\M$ and recall that $\pi$ has the approximate unique extension property by Proposition \ref{P:AUEPSum}. Thus, if we define \[\Lambda:\rC^*(\text{Im}\psi)\rightarrow \text{Im}\pi, \ \ \ \ \ \ (x_i)_{i\in I}\mapsto(\Gamma_{\psi_i}(x_i))_{i\in I},\]then observe that $\Lambda$ is a completely positive map satisfying $\Lambda\circ\psi = \pi.$ So we have that $\Lambda = \Gamma_\psi$ by Theorem \ref{T:GammaUEP} {\rm (i)}.
\end{proof}

Next, we address how the conditional expectation $\Gamma_\psi$ is affected by approximately unitarily equivalent completely positive extensions. We say that unital completely positive maps $\psi,\psi':\rC^*(\M)\rightarrow B(\H_\pi)$ are \emph{approximately unitarily equivalent} whenever there is a net of unitary operators $w_\gamma\in B(\H_\pi)$ satisfying \[ \lim_\gamma \| w_\gamma^*\psi'(t)w_\gamma - \psi(t)\| = 0, \ \ \ \ \ \ \ t\in\rC^*(\M).\]In particular, when $\psi$ and $\psi'$ are approximately unitarily equivalent, there is a $*$-isomorphism $\Omega:\rC^*(\text{Im}\psi)\rightarrow\rC^*(\text{Im}\psi')$ satisfying $\Omega\circ\psi = \psi'$.

\begin{corollary}\label{C:UniqueNorm}
Let $\M\subset B(\H)$ be a unital operator space. Suppose that $\pi:\rC^*(\M)\rightarrow B(\H_\pi)$ is a $*$-representation possessing the approximate unique extension property and let $\psi,\psi': \rC^*(\M)\rightarrow B(\H_\pi)$ be unital completely positive maps satisfying $\psi\mid_\M = \psi'\mid_\M = \pi\mid_\M$. Assume that there is a $*$-isomorphism $\Omega:\rC^*(\text{Im}\psi)\rightarrow\rC^*(\text{Im}\psi')$ satisfying $\Omega\circ\psi = \psi'.$ Then, we have that $\rC^*(\text{Im}\psi) = \rC^*(\text{Im}\psi')$ and $\Gamma_\psi = \Gamma_{\psi'}$.
\end{corollary}

\begin{proof}
Observe that, by Theorem \ref{T:AUEPsequence} {\rm (i)}, we have that\[\pi = \Gamma_{\psi'}\circ\psi' = \Gamma_{\psi'}\circ\Omega\circ\psi.\]By Theorem \ref{T:GammaUEP} {\rm (i)}, it follows that $\Gamma_{\psi'}\circ\Omega = \Gamma_\psi$. Whence, we obtain that $\ker\Gamma_{\psi'}=\ker\Gamma_\psi$ on account of $\Omega$ being a $*$-isomorphism. By Corollary \ref{C:Correspondence}, we conclude that $\rC^*(\text{Im}\psi) = \rC^*(\text{Im}\psi')$ and $\Gamma_\psi = \Gamma_{\psi'}$. \end{proof}

We close this section by providing an example of when the approximate unique extension property is equivalent to the usual unique extension property. This allows us to improve Theorem \ref{T:AUESummary} in a special case.

\begin{theorem}\label{T:Compact}
Let $\M\subset B(\H)$ be a unital operator space and let $\pi: \rC^*(\M)\rightarrow B(\H_\pi)$ be a $*$-representation such that $\fK(\H)\subset \text{Im}\pi$. Then, the following statements are equivalent.\begin{enumerate}[\rm (i)]
\item $\pi$ possesses the unique extension property.
\item $\pi$ is approximately unitarily equivalent to a boundary representation.
\item $\pi$ possesses the approximate unique extension property.
\end{enumerate}
In particular, if $\M$ is separable and $\rC^*(\M)$ is postliminal, then every irreducible $*$-representation of $\rC^*(\M)$ is a boundary representation if and only if every $*$-representation of $\rC^*(\M)$ possesses the approximate unique extension property.
\end{theorem}

\begin{proof}
The direction {\rm (i)} $\Rightarrow${\rm (ii)} is trivial and {\rm (ii)}$\Rightarrow${\rm (iii)} is Theorem \ref{T:UEPimplyAUEP}. For {\rm (iii)}$\Rightarrow${\rm (i)}, let $\psi:\rC^*(\M)\rightarrow B(\H_\pi)$ be a completely positive extension of $\pi\mid_\M$. Then, observe that $\Gamma_\psi$ is a $*$-representation that fixes $\fK(\H)$ and consequently, must be the identity representation (for example, see \cite[Proposition 3.5]{paulsen2011weak}). In which case, by Theorem \ref{T:AUEPsequence} {\rm (v)}, we have that $\psi = \pi$ and so $\pi$ has the unique extension property. 

For the last statement, if every irreducible $*$-representation of $\rC^*(\M)$ is a boundary representation, then every $*$-representation of $\rC^*(\M)$ possesses the approximate unique extension property by Theorem \ref{T:AUESummary}. For the converse, note that every irreducible $*$-representation $\sigma:\rC^*(\M)\rightarrow B(\H_\sigma)$ possesses the approximate unique extension property. Since $\rC^*(\M)$ is postliminal, we have that $\fK(\H_\sigma)\subset \text{Im}\sigma$. Therefore, the previous paragraph guarantees that $\sigma$ has the unique extension property.\end{proof}

We remark that the last statement within Theorem \ref{T:Compact} allows us to rephrase Arveson's hyperrigidity conjecture in a particular case. Let $\M\subset B(\H)$ be a separable unital operator space such that $\rC^*_e(\M)$ is postliminal and suppose that there is a $*$-representation of $\rC^*_e(\M)$ that does not possess the unique extension property. Then, for Arveson's conjecture to be true, there must be a $*$-representation of $\rC^*_e(\M)$ that does not possess the approximate unique extension property.

Finally, note that Theorem \ref{T:Compact} allows us to readily identify many instances where $*$-representations with the approximate unique extension property will necessarily factor through the $\rC^*$-envelope.

\begin{example}\label{E:DA}
Let $H^2$ denote the classical Hardy space on the open unit disc $\bD\subset\bC$ and let $\bT\subset\bC$ denote the unit circle. We let $\rA(\bD)$ denote the disc algebra, which is the closed subalgebra of holomorphic functions in $\rC(\ol{\bD})$.

Consider the subalgebra $\A\subset B(H^2)$ of multiplication operators $M_\varphi$ where $\varphi\in\rA(\bD)$. The reader may consult \cite{agler2002pick},\cite{douglas2012banach},\cite{hoffman2007banach} for more details on this algebra. For our purposes, we note that there is a well-known split exact sequence of $\rC^*$-algebras \[\begin{tikzcd}
0 \arrow{r} & \fK(H^2) \arrow{r} & \rC^*(\A) \arrow{r} & \rC(\bT) \arrow{r} & 0\end{tikzcd}\]and that $\rC^*(\A)$ is the Toeplitz algebra. Moreover, we have that $\rC^*_e(\A)\cong\rC(\bT)$ (for example, this follows by \cite[Corollary 6.4]{clouatre2018multiplier}).

Note that the image of the identity representation $\pi:\rC^*(\A)\rightarrow B(H^2)$ contains $\fK(H^2)$. Since $\pi$ does not factor through the $\rC^*$-envelope of $\A$, we have that $\pi$ cannot be a boundary representation for $\A$ by Theorem \ref{T:HR} {\rm (ii)}. So $\pi$ does not possess the approximate unique extension property either by Theorem \ref{T:Compact}.
\end{example}

\section{Concrete applications}\label{S:Examples}

In this section, we apply our work to a couple classes of operator spaces. This allows us to reinterpret the classical \v{S}a\v{s}kin Theorem (Proposition \ref{P:Saskin}) and obtain new formulations of Arveson's essential normality conjecture (Theorem \ref{T:HREssNorm}).

\subsection{Function spaces}\label{SS:Function}

Here, we study the approximate unique extension property in the presence of commutativity. For this, we recall the classical theorem of \v{S}a\v{s}kin \cite{vsavskin1967mil}. Given a compact metric space $X$ and a unital function space $\M\subset\rC(X)$, \v{S}a\v{s}kin proved that every irreducible $*$-representation of $\rC(X)$ is a boundary representation for $\M$ if and only if, whenever $\psi_n:\rC(X)\rightarrow\rC(X)$ is a sequence of unital completely positive maps satisfying $\| \psi_n(m)-m\|\rightarrow0$ for every $m\in\M$, then we have that $\| \psi_n(t)-t\|\rightarrow0$ for every $t\in\rC(X)$. A non-separable extension also holds upon replacing sequences with nets \cite[Theorem 5.3]{davidson2021choquet}.

\begin{proposition}\label{P:Saskin}
Let $X$ be a compact Hausdorff space and $\M\subset\rC(X)$ be a unital function space. Assume that $\rC(X)$ is the $\rC^*$-envelope of $\M$. Then, the following statements are equivalent.\begin{enumerate}[\rm (i)]
\item Whenever $\pi:\rC(X)\rightarrow B(\H_\pi)$ is a $*$-representation possessing the approximate unique extension property and $\psi:\rC(X)\rightarrow B(\H_\pi)$ is a unital completely positive map such that $\rC^*(\text{Im}\psi)$ is commutative and $\psi\mid_\M=\pi\mid_\M$, we have that $\psi = \pi$.
\item Whenever $\psi_\alpha:\rC(X)\rightarrow\rC(X)$ is a net of unital completely positive maps satisfying $\| \psi_\alpha(g)-g\|\rightarrow0$ for every $g\in\M$, then we have that $\| \psi_\alpha(f)-f\|\rightarrow0$ for every $f\in\rC(X)$.
\item Every irreducible $*$-representation for $\rC(X)$ is a boundary representation. 
\end{enumerate}
\end{proposition}

\begin{proof}
{\rm (ii)}$\Leftrightarrow${\rm (iii)}: This is the non-separable version of \v{S}a\v{s}kin's Theorem \cite[Theorem 5.3]{davidson2021choquet}.

{\rm (iii)}$\Rightarrow${\rm (i)}: Let $\pi:\rC(X)\rightarrow B(\H_\pi)$ be a $*$-representation possessing the approximate unique extension property and let $\psi:\rC(X)\rightarrow B(\H_\pi)$ be a unital completely positive map such that $\rC^*(\text{Im}\psi)$ is commutative and $\psi\mid_\M=\pi\mid_\M$. If $Y$ is a compact Hausdorff space and $\Omega:\rC^*(\text{Im}\psi)\rightarrow\rC(Y)$ is $*$-isomorphism, then $\Omega\circ\psi:\rC(X)\rightarrow \rC(Y)$ is a unital completely positive map satisfying $(\Omega\circ\psi)\mid_\M = (\Omega\circ\pi)\mid_\M.$ Thus, $\Omega\circ\psi=\Omega\circ\pi$ by \cite[Theorem 5.1]{davidson2021choquet} and so, $\psi = \pi$ as desired.

{\rm (i)}$\Rightarrow${\rm (ii)}: Let $\pi:\rC(X)\rightarrow B(\H_\pi)$ be an isometric $*$-representation and let $\psi_\alpha:\rC(X)\rightarrow\rC(X)$ be a net of unital completely positive maps satisfying $\| \psi_\alpha(g)-g\|\rightarrow0$ for every $g\in\M$. Define a $*$-homomorphism\[\pi_\U : \rC(X)\rightarrow\left( \prod_\alpha \text{Im}\pi\right)/\fJ, \ \ \ \ \ \ \ f\mapsto(\pi(f))_\alpha +\fJ,\]where $\fJ$ is the closed two-sided ideal of $\prod_\alpha \text{Im}\pi$ consisting of those nets $(\pi(f_\alpha))_\alpha$ such that $\|\pi(f_\alpha)\| = \| f_\alpha\|\rightarrow0$. By definition of $\fJ$, we have that $\ker\pi_\U = \ker\pi$ and so $\pi_\U$ is isometric. Upon representing $(\prod_\alpha \text{Im}\pi)/\fJ$ on some Hilbert space via an isometric $*$-homomorphism $\iota$, we may apply Proposition \ref{P:AUEPker} {\rm (ii)} to obtain a cardinal $\kappa$ such that $(\iota\circ\pi_\U)^{(\kappa)}$ has the approximate unique extension property.

Now, define a unital completely positive map \[ \psi:\rC(X)\rightarrow\left( \prod_\alpha \text{Im}\pi\right)/\fJ, \ \ \ \ \ \ \ f\mapsto ((\pi\circ\psi_\alpha)(f))_\alpha+\fJ.\]By assumption, we have that $\psi$ and $\pi_\U$ agree on $\M$. Therefore, we have that $(\iota\circ\psi)^{(\kappa)}$ and $(\iota\circ\pi_\U)^{(\kappa)}$ agree on $\M$ as well. However, since $\rC^*(\text{Im}(\iota\circ\psi)^{(\kappa)})$ is a commutative $\rC^*$-algebra, the assumption allows us to conclude that $(\iota\circ\psi)^{(\kappa)} = (\iota\circ\pi_\U)^{(\kappa)}$. As $\iota$ and $\pi_\U$ are isometric, and by definition of $\pi_\U$, it then follows that $\| \psi_\alpha(f)- f\|\rightarrow 0$ for every $f\in\rC(X).$
\end{proof}

We remark that Proposition \ref{P:Saskin} is very similar to \cite[Theorem 5.1]{davidson2021choquet}, but in condition {\rm (i)} we may make the additional assumption that $\pi$ has the approximate unique extension property. Due to Theorem \ref{T:Compact}, we see that (a priori) the only non-trivial implication in condition {\rm (i)} is in regards to those $*$-representations that fail to be irreducible. On the other hand, Proposition \ref{P:Saskin} guarantees that this condition is in fact equivalent to a statement about irreducible $*$-representations.

Finally, we would like to note that it also appears somewhat notable that there are many examples of unital operator spaces where condition {\rm (i)} fails (Example \ref{E:all_BR_aue}).

\subsection{Essential normality}\label{SS:EssNorm}

Next, we apply our machinery to a different class of operator systems and relate our work to another well-studied conjecture of Arveson.

For these purposes, let $d\geq2$ be a fixed positive integer, $\bB_d\subset\bC^d$ denote the open unit ball and $\bC[z] = \bC[z_1, \ldots, z_d]$ denote the algebra of complex $d$-variate polynomials. The Drury-Arveson space $H_d^2$ is the reproducing kernel Hilbert space on $\bB_d$ that is associated with the kernel \[k(z,w) = \frac{1}{1-\langle z, w\rangle}, \ \ \ \ \ \ z,w\in\bB_d.\]

Let $M_p\in B(H_d^2)$ denote the polynomial multiplier associated with $p\in\bC[z]$. The row operator \[M_z = (M_{z_1}, \ldots, M_{z_d}): (H_d^2)^{(d)} \longrightarrow H_d^2\] is called the $d$-shift of the Drury-Arveson space. Next, fix a homogeneous ideal $I\vartriangleleft\bC[z]$. It is clear that $IH_d^2$ is invariant for $M_z$. Then, we define a $d$-tuple $S = (S_1, \ldots, S_d)$ of operators on $\F_I = H_d^2\ominus IH_d^2$ given by \[ S_i = P_{\F_I} M_{z_i}\mid_{\F_I}, \ \ \ \ \ \ 1\leq i \leq d.\]Define $\fT_I = \rC^*(\S_I)$ to be the Toeplitz algebra associated with the homogeneous ideal $I$. Then, it is known that $\fK(\F_I)\subset \fT_I$ \cite[Theorem 1.3]{popescu2006operator} and so we define $\O_I = \fT_I/\fK(\F_I)$ to be the Cuntz algebra associated with the ideal. Throughout, we let $q:\fT_I\rightarrow\O_I$ denote the quotient map. We now present the main conjecture on this tuple of operators.

$\left.\right.$\\
{\bf Arveson's Essential Normality Conjecture:} If $I\vartriangleleft\bC[z]$ is a homogeneous ideal, then $S_iS_j^*-S_j^*S_i$ is a compact operator for every $1\leq i,j \leq d.$\\
\vspace{1mm}

Tuples of operators that satisfy the above criterion are said to be \emph{essentially normal}. Among the main interests in Arveson's essential normality conjecture is that $S$ is the universal operator $d$-tuple for commuting row contractions which satisfy $p(S) = 0$ for each $p\in I$ \cite[Theorem 8.4]{shalit2009subproduct}. Accordingly, the essential normality conjecture is connected with a theory of non-commutative varieties on the ball.

Kennedy and Shalit revealed a connection between the essential normality conjecture and Arveson's hyperrigidity conjecture. Assume that $I\vartriangleleft\bC[z]$ is a non-zero homogeneous ideal that contains no linear polynomials and is not of finite codimension in $\bC[z]$. If we let $\S_I$ denote the operator system generated by $S_1, \ldots, S_d$, then essential normality of the tuple $S$ is equivalent to hyperrigidity of the operator system $\S_I$ \cite[Theorem 4.12]{kennedy2015essential},\cite{kennedy2016corrigendum}. Now, we will reframe essential normality of the tuple in terms of the approximate unique extension property.

\begin{theorem}\label{T:HREssNorm}
Let $I\vartriangleleft\bC[z]$ be a non-zero homogeneous ideal that contains no linear polynomials and such that $I$ is not of finite codimension in $\bC[z]$. Then, the following statements are equivalent.\begin{enumerate}[\rm (i)]
\item There is an isometric $*$-representation of $\O_I$ that has the approximate unique extension property with respect to $q(\S_I)$.
\item The tuple $S= (S_1, \ldots, S_d)$ is essentially normal.
\item The operator system $\S_I$ is hyperrigid.
\item Every irreducible $*$-representation of $\fT_I$ is a boundary representation for $\S_I$.
\end{enumerate}
\end{theorem}

\begin{proof}
The direction {\rm (iii)}$\Rightarrow$ {\rm (iv)} is trivial, {\rm (iv)}$\Rightarrow${\rm (i)} is a consequence of Lemma \ref{L:ApproxUEPquotient} and Theorem \ref{T:AUESummary}, and {\rm (ii)}$\Leftrightarrow${\rm (iii)} is \cite[Theorem 4.12]{kennedy2015essential},\cite{kennedy2016corrigendum}.

{\rm (i)}$\Rightarrow$ {\rm (ii)}: Let $\pi:\O_I\rightarrow B(\H_\pi)$ be an isometric $*$-representation that has the approximate unique extension property with respect to $q(\S_I)$ and consider $\sigma:=\pi\circ q$. Observe that $\sigma$ vanishes over the compact operators on $\F_I$ and so, by \cite[Proposition 4.13]{kennedy2015essential}, there exists a unital completely positive map $\psi:\fT_I\rightarrow B(\H_\pi)$ such that $\psi\mid_{\S_I} = \sigma\mid_{\S_I}$ and \[\psi(S_i^*S_j)  = \sigma(S_jS_i^*), \ \ \ \ \ \ 1\leq i,j\leq d.\]Moreover, by \cite[Lemma 4.6]{kennedy2015essential}, we may express $\psi = \varphi\circ q$ where $\varphi:\O_I\rightarrow B(\H_\pi)$ is a unital completely positive map. Since $\varphi\mid_{q(\S_I)} = \pi\mid_{q(\S_I)}$, we obtain a net of unitaries $u_\beta\in B(\H_\pi)$ such that $u_\beta^*\varphi(x)u_\beta\rightarrow \pi(x)$ in the weak operator topology for every $x\in\O_I$. Therefore, $u_\beta^*\psi(t)u_\beta\rightarrow\sigma(t)$ in the weak operator topology for every $t\in\fT_I$. As in the proof of Theorem \ref{T:AUEPsequence}, we obtain a homomorphic conditional expectation $\Gamma_\psi: \rC^*(\text{Im}\psi)\rightarrow\text{Im}\sigma$ satisfying $\Gamma_\psi\circ\psi = \sigma$. In particular, we see that \[ \sigma(S_i^*S_j) = (\Gamma_\psi\circ \psi)(S_i^*S_j) = \sigma(S_jS_i^*), \ \ \ \ \ \ 1\leq i,j\leq d.\]Since $\ker\sigma = \fK(\F_I)$, we obtain that $S_i^*S_j-S_jS_i^*$ is compact for every $1\leq i,j\leq d$. In other words, the tuple $S = (S_1,\ldots, S_d)$ is essentially normal.
\end{proof}

\begin{remark}\label{R:EssNorm}
Mimicing the above proof, one may replace condition (i) by the statement that there is a $*$-representation of $\fT_I$ possessing the approximate unique extension property with respect to $\S_I$ with kernel equal to the compact operators on $\F_I$.
\end{remark}

The fact that the operator system $\S_I$ satisfies the conditions of Arveson's hyperrigidity conjecture had been previously unrecorded (see \cite[Discussion after Theorem 4.12]{kennedy2015essential}). However, one can derive this from the results in \cite{kennedy2015essential}. Indeed, by \cite[Theorem 3.3]{kennedy2015essential}, the $\rC^*$-envelope of $q(\S_I)$ is commutative and so the boundary representations consist of some collection of characters. Recall that boundary representations must factor through the $\rC^*$-envelope by Theorem \ref{T:HR} {\rm (ii)}. It is then easy to check that, if a boundary representation $\pi:\fT_I\rightarrow B(\H_\pi)$ for $\S_I$ vanishes over the compact operators on $\F_I$, then $\pi$ must be a character. Thus, if every irreducible $*$-representation of $\fT_I$ is a boundary representation, we must have that $\O_I$ is a commutative $\rC^*$-algebra. Consequently, the tuple $S$ is essentially normal and thus, $\fT_I$ is hyperrigid by \cite[Theorem 4.12]{kennedy2015essential},\cite{kennedy2016corrigendum}.

A priori, condition {\rm (i)} appears to be a weak stipulation and thus, one may expect that there are many choices of homogeneous ideals that will result in essentially normal tuples. However we remind the reader that, since essential normality of the tuple $S$ is equivalent to $\O_I$ being the $\rC^*$-envelope of $q(\S_I)$, it is unclear whether $\O_I$ will necessarily have an isometric $*$-representation with the approximate unique extension property.

\bibliographystyle{plain}
\bibliography{approx_uep_arXiv_v1}
	
\end{document}